\documentclass[12pt,leqno]{amsart}%
\usepackage{amsmath}
\usepackage[margin=1.5in]{geometry}
\usepackage{txfonts}
\usepackage{graphicx}
\usepackage{amsfonts}
\usepackage{amssymb}
\usepackage[utf8]{inputenc}
\usepackage{comment}
\usepackage{hyperref}
\usepackage{xcolor}
\usepackage{lipsum}%
\providecommand{\U}[1]{\protect\rule{.1in}{.1in}}
\newtheorem{theorem}{Theorem}[section]
\theoremstyle{plain}

\newtheorem{corollary}{Corollary}[section]

\newtheorem{lemma}{Lemma}[section]

\newtheorem{proposition}{Proposition}[section]

\numberwithin{equation}{section}
\theoremstyle{definition}
\newtheorem{definition}{Definition}
\newtheorem{remark}{Remark}[section]

\begin{document}
\title[Hamiltonian Stationary equations]{Regularity of Hamiltonian Stationary Equations in Symplectic manifolds}
\author{Arunima Bhattacharya, Jingyi Chen, and Micah Warren}
\address{Department of Mathematics\\
	University of Washington, Seattle, WA 98195}
\email{arunimab@uw.edu}

\address{Department of Mathematics\\
	 University of British Columbia, Vancouver,
	BC V6T 1Z2}
\email{jychen@math.ubc.ca}

\address{Department of Mathematics\\
	University of Oregon, Eugene, OR 97403}
\email{micahw@uoregon.edu}

\maketitle

\begin{abstract}
In this paper, we prove that any $C^{1}$-regular Hamiltonian stationary
Lagrangian submanifold in a symplectic manifold is smooth. More broadly, we
develop a regularity theory for a class of fourth order nonlinear elliptic
equations with two distributional derivatives. Our fourth order regularity
theory originates in the geometrically motivated variational problem for the volume
functional, but should have applications beyond.

\end{abstract}

\section{Introduction}

The main purpose of this paper is to prove the assertion: Any $C^{1}$-regular
Hamiltonian stationary Lagrangian submanifold in a symplectic manifold is
smooth. \let\thefootnote\relax\footnotetext{Chen was partially supported by an
NSERC Discovery Grant (22R80062).}

We achieve this by developing a regularity theory for a class of fourth order
nonlinear equations of double divergence form
\begin{equation}
\label{eqn:Intro}\partial_{x_{l}} \partial_{x_{j}}F^{jl}(x,Du,D^{2}u) =
\partial_{x_{k}} a^{k}(x,Du,D^{2}u)-b(x,Du,D^{2}u).
\end{equation}
The coefficient functions $F^{jl},a^{k},b$ are smooth in the entries
$(x,Du,D^{2}u)$ over a convex region $U\subset\mathbb{R}^{n}\times
\mathbb{R}^{n}\times S^{n\times n}$, and the Legendre ellipticity condition
holds: for a constant $\Lambda>0$
\begin{equation}
\label{elliptic:Intro}\frac{\partial F^{jl}}{\partial u_{ik}}(\xi)\sigma
_{ij}\sigma_{kl}\geq\Lambda\left\Vert \sigma\right\Vert ^{2},\text{ $\forall$
}\sigma\text{ $\in S^{n\times n}$ and $\xi\in U$}.
\end{equation}
A function $u\in W^{2,\infty}$ is said to be a weak solution to the double
divergence equation \eqref{eqn:Intro} if each of the derivatives
$\partial_{x_{i}}$ presented in \eqref{eqn:Intro} are taken in a
distributional sense, as in \eqref{main}. For non-classical solutions to
nonlinear partial differential equations, especially of order beyond two,
attention needs to be paid even for the meaning of solutions, due to the fact
that no uniform theory exists. In our case, the double divergence structure on
the matrix-valued operator $F$, which involves $D^{2}u$ itself, permits us to
define solutions, possibly in the weakest form, by flipping derivatives on $F$
and the lower order terms, to test functions via integration by parts as
traditionally done for distributional solutions, but now only for half of the
total order.

Equations in divergence form occupy an important place in the second order
PDE theory. In fourth order, the most natural counterpart is an equation, linear or nonlinear, with a double divergence structure. Many well-known equations enjoy the structure such as for the bi-harmonic functions, extremal K\"ahler metrics, the Willmore surface, and the Hamiltonian stationary Lagrangian equations which are closely linked to elastic mechanics. We find that the double divergence structure, a less explored area, shares similar features, as second order equations in divergence form, toward a regularity theory. We demonstrate that when
\eqref{elliptic:Intro} holds, any weak solution $u$ to \eqref{eqn:Intro} is
smooth, provided that the oscillation of $D_{q}F(x,Du,D^{2}u)$ can be bounded
locally (in $x$) by a small positive constant.

The above fourth order nonlinear elliptic equation originates in the
variational problem for volume of Lagrangian submanifolds under Hamiltonian
variations in a symplectic manifold $(M,\omega)$ with a Riemannian metric $g$
compatible with $\omega$ in the sense that $\omega(X,Y) =g(JX,Y)$ for an
almost complex structure $J$ on $M$.

A Lagrangian submanifold $L$ is Hamiltonian stationary if its mean curvature
1-form $\omega(H, \cdot)$ is closed and coclosed, i.e. a harmonic 1-form on
$L$ w.r.t. the induced metric from $(M,g)$ (cf. Oh \cite{Oh} also see
\cite[p.1071-1072]{JLS}). In a Calabi-Yau
manifold $(M,\omega,\Omega)$ of complex dimension $n$, this is further equivalent to a scalar equation: the Lagrangian phase function $\Theta$ is harmonic.  Here the holomorphic
$n$-form $\Omega$ satisfies $\Omega\wedge\overline{\Omega}={\omega^{n}%
}/{n!}$ and defines $\Theta$ by $\Omega|_{L} = e^{\sqrt{-1}\Theta}d\mu_{L}$.  
The scalar equation follows from the relation $H=J\nabla\Theta$ (\cite{HL82}, \cite{Oh}, \cite{SWJDG}).

In $\mathbb{C}^{n}$ with the standard K\"ahler structure, a particular
expression for $\Theta$ is available, namely, it is a sum of arctan of the
eigenvalues of the Hessian of the potential function $u$ for a local graphical
representation $L = (x,Du)$. This decomposition feature of the fourth order
operator into two second order elliptic operators is essential in the work of
Chen-Warren \cite{CW} in which it is shown that a $C^{1}$-regular Hamiltonian
stationary Lagrangian submanifold in $\mathbb{C}^{n}$ is real analytic.
However, the same strategy for a Calabi-Yau other than $\mathbb{C}^{n}$
encounters difficulties for the reason that $\Theta$, still well-defined by
$\Omega$ at least locally, now is no longer written in a clean form as sum of
arctan functions, when representing $L$ as a gradient graph in a Darboux coordinate chart.

To overcome the obstacle presented above in the Calabi-Yau case, we find that,
in a more general standpoint, the Riemannian picture without referring to a
symplectic structure is helpful: dealing directly with the stationary point of
the volume of $L=(x,Du)$ in an open ball $B\subset\mathbb{R}^{2n}$ equipped
with a Riemannian metric among nearby competing gradient graphs $L_{t}=(x, Du
+t D\eta)$ for compactly supported smooth functions $\eta$. This leads us to
study the fourth order nonlinear equation \eqref{eqn:Intro} with \eqref{elliptic:Intro}.

We now outline our approach to the regularity problem. Given a $W^{2,\infty}$
weak solution $u$ of \eqref{eqn:Intro} that satisfies the Legendre ellipticity
condition \eqref{elliptic:Intro}, we show, in Proposition \ref{reg3}, that the
difference quotient $[u(x)-u(x-h)]/|h|$ can be bounded in $W^{2,2}$ uniformly
in $h$. Letting $h\to0$ asserts $u\in W^{3,2}$ with estimates controlled by
$\|u\|_{W^{2,\infty}}$. This boosted regularity is then used to bound the
$C^{1,\alpha}$ norm of the difference quotient uniformly in $h$ in Proposition
\ref{reg}, leading to a $C^{2,\alpha}$ bound on $u$. The key ingredient for
this step is a closeness assumption, given by (\ref{cons}): this ensures that
the operator is in fact close to a constant coefficient operator, given by its
linearization at the origin, that leads to a uniform $C^{1,\alpha}$ bound on
the difference quotient. Note that reaching $C^{2,\alpha}$ is a crucial step in
proving smoothness since once $C^{2,\alpha}$ is achieved
the functions $\frac{\partial F^{jl}}{\partial u_{ik}},
\frac{\partial F^{jl}}{\partial u_{k}},\frac{\partial F^{jl}}{\partial x_{p}}
$, which were barely measurable, are now all H\"older continuous in $x$, and this 
is sufficient to prove higher regularity for the equation satisfied by the difference quotient. The
enhanced regularity alone improves the bound on the difference between the
actual operator and its linearization by a factor of a power of $r$, which in
turn ultimately leads to $u\in C^{3,\alpha}$. Moving from $C^{3,\alpha}$ to
$C^{\infty}$ involves a similar bootstrapping procedure employed in
\cite{BW1} by considering the difference quotient.

For the general fourth order nonlinear equation, our main result is the following.

\begin{theorem}
\label{main1:Intro} Suppose that $u\in W^{2,\infty}(B_{1})$ is a weak solution
of \eqref{eqn:Intro} that satisfies condition \eqref{elliptic:Intro} on the
unit ball $B_{1}$ in $\mathbb{R}^{n}$. There is an $\varepsilon_{0}%
(\Lambda,n)>0$
such that if
\begin{equation}
\left|  \frac{\partial F^{jl}}{\partial u_{ik}}(x, Du, D^{2}u) -
\frac{\partial F^{jl}}{\partial u_{ik}}(\xi)\right|  <\varepsilon_{0}
\label{cons1}%
\end{equation}
for some $\xi\in U$ and all $x\in B_{1}$, then $u$ is smooth in $B_{1}$.
\end{theorem}


This regularity statement suffices for answering affirmatively the motivating
geometric question on smoothness of a $C^{1}$-regular critical point under
Hamiltonian deformations in a symplectic manifold. The transition, from the
general theory in euclidean space to the specific symplectic setting, is done
in a Darboux coordinate chart with estimates on the Riemannian metric within the
special coordinates. This is given by \cite[Prop. 3.2 and Prop. 3.4]{JLS}. Our
main result is the following.

\begin{theorem}
\label{main2:Intro} Let $(M,\omega)$ be a compact symplectic manifold with a
Riemannian metric $g$ compatible with $\omega$ and some almost complex
structure $J$ on $M$. Let $L$ be a Hamiltonian stationary Lagrangian $C^{1}%
$-regular submanifold in $M$ with respect to $\omega,g$. Then $L$ is smooth.
\end{theorem}

As critical points of the volume functional on submanifolds, Theorem
\ref{main2:Intro} may be compared to some of the
classical statements for minimal submanifolds. For minimal submanifolds
(stationary for all smooth variations with compact support), a classical
theorem of Morrey states: $C^{1}$-regular minimal submanifolds are smooth
\cite[Theorem 10.7.1]{MorreyBook}. On the other hand, Lawson-Osserman
\cite{LO} constructed enlightening examples demonstrating existence of
Lipschitz minimal submanifolds (even graphical) that are not $C^{1}$. More
generally, the regularity theory developed in \cite{GM} for second order
elliptic systems does not seem to have direct impact on our single equation of
higher order on a scalar function.

Convergence of a sequence of Hamiltonian stationary Lagrangian submanifolds
was studied by Chen-Warren in \cite{CW2}, the analysis therein, especially the
smoothness estimates and $\varepsilon$-regularity, requires decomposing the
fourth order operator into the form dealt with in \cite{CW}, therefore only
established for $\mathbb{C}^{n}$. For a general K\"ahler background,
techniques special to surfaces, such as conformality and bubble tree
convergence with roots in the development of minimal surfaces, harmonic maps
and $J$-holomorphic curves, were used to prove compactness statements in
Chen-Ma \cite{ChenMa} and Schoen-Wolfson \cite{SW}. In light of the new
treatment about regularity in this paper, we will investigate the compactness
question in a K\"ahler manifold of any dimension in a forthcoming paper.

The organization of the paper is as follows: in section \ref{sec_fourth}, we
introduce in detail the class of fourth order nonlinear equations and develop
a regularity theory. In section \ref{sec_EL}, we derive the Euler-Lagrange
equations for the variational problem on a Riemannian ball and show that it
takes the form of the fourth order equation discussed in section
\ref{sec_fourth}. Finally, in section \ref{sec_ main}, we prove regularity for
Hamiltonian stationary Lagrangian submanifolds in a symplectic manifold.

\bigskip

\textbf{Notations.} Through out this paper, we use $B_{r}$ to denote a ball
with radius $r$ and center at the origin in $\mathbb{R}^{n}$, unless specified otherwise.

\section{Fourth order elliptic theory}


\label{sec_fourth}

\subsection{Preliminaries}

\label{pre} We consider the following fourth order equation, written in double
divergence form:
\begin{equation}
\int_{B_{1}}\left[  F^{jl}(x,Du,D^{2}u)\eta_{jl}+a^{k}(x,Du,D^{2}u)\eta
_{k}+b(x,Du,D^{2}u)\eta\right]  dx=0 \label{main}%
\end{equation}
for all $\eta\in C_{c}^{\infty}(B_{1})$ where $B_{1}$ is the unit ball in
$\mathbb{R}^{n}$. The coefficients are smooth in the entries $(x,Du,D^{2}u)$
over a given convex region $U\subset\mathbb{R}^{n}\times\mathbb{R}^{n}\times
S^{n\times n}.$ Lower indices on a function stand for partial derivatives,
e.g. $\eta_{jl},\eta_{k}$, and summation convention is assumed. \newline

\ We write $h_{p}=he_{p}$ and denote the difference quotient of $u$ in the
$e_{p}$ direction by $u^{h_{p}}$. We start by deriving a difference quotient
expression from (\ref{main}) in the direction $h_{p}$. Fixing a compactly
supported function $\eta$ we can choose $h$ small enough so the function
\begin{equation}
\eta^{-h_{p}}(x)=\frac{\eta(x-h_{p})-\eta(x)}{h} \label{etad}%
\end{equation}
is a valid test function. Using a change of variables $x\rightarrow x+h_{p}$
on the first term of (\ref{etad}) with the first two terms of (\ref{main}) and
recombining, we get%
\begin{equation}
\int_{B_{1}}\bigg([F^{jl}(x,Du,D^{2}u)]^{h_{p}}\eta_{jl}+a^{k}(x,Du,D^{2}%
u)\eta_{k}^{-h_{p}}+b(x,Du,D^{2}u)\eta^{-h_{p}}\bigg)dx=0. \label{diff1}%
\end{equation}

The function $F^{jl}$ is defined on open subsets of the vector space so for
any fixed $x$ where $D^2u(x)$ is defined we can define%

\begin{align*}
\xi_{0}  &  =\left(  x,Du(x),D^{2}u(x)\right)  \in%
\mathbb{R}
^{n}\times%
\mathbb{R}
^{n}\times S^{n\times n}\\
\xi_{h}  &  =\left(  x+h_{p},Du(x+h_{p}),D^{2}u(x+h_{p})\right)  \in%
\mathbb{R}
^{n}\times%
\mathbb{R}
^{n}\times S^{n\times n}\\
\vec{V}  &  =\xi_{h}-\xi_{0}%
\end{align*}
in which case we have%
\begin{align*}
\lbrack F^{jl}  &  (x,Du,D^{2}u)]^{h_{p}}=\frac{1}{h}\{F^{jl}(\xi_{0}+\vec
{V})-F^{jl}(\xi_{0})\}\\
&  =\frac{1}{h}\int_{0}^{1}\frac{d}{dt}F^{jl}(\xi_{0}+t\vec{V})dt\\
&  =\frac{1}{h}\int_{0}^{1}DF^{jl}|_{\xi_{0}+t\vec{V}}\cdot\vec{V}dt\\
&  =\int_{0}^{1}\frac{\partial F^{jl}}{\partial u_{ik}}({\xi_{0}+t\vec{V}%
})\cdot u_{ik}^{h_{p}}dt+\int_{0}^{1}\left(  \frac{\partial F^{jl}}{\partial
u_{k}}({\xi_{0}+t\vec{V}})u_{k}^{h_{p}}+\frac{\partial F^{jl}}{\partial x_{p}%
}(\xi_{0}+t\vec{V})\right)  dt\\
&  =\left(  \int_{0}^{1}\frac{\partial F^{jl}}{\partial u_{ik}}({\xi_{0}%
+t\vec{V}})dt\right)  \cdot u_{ik}^{h_{p}}+\int_{0}^{1}\left(  \frac{\partial
F^{jl}}{\partial u_{k}}({\xi_{0}+t\vec{V}})u_{k}^{h_{p}}+\frac{\partial
F^{jl}}{\partial x_{p}}(\xi_{0}+t\vec{V})\right)  dt\\
&  =\beta^{ij,kl}\cdot u_{ik}^{h_{p}}+\gamma_{1}^{jl,k}u_{k}^{h_{p}}%
+\gamma_{2}^{jl}%
\end{align*}
where we define
\begin{equation}
\beta^{ij,kl}(x)=\int_{0}^{1}\frac{\partial F^{jl}}{\partial u_{ik}}(\xi
_{0}+t\vec{V})dt \label{bterm}%
\end{equation}
and
\begin{align}
\gamma_{1}^{jl,k}(x)  &  =\int_{0}^{1}\frac{\partial F^{jl}}{\partial u_{k}%
}({\xi_{0}+t\vec{V}})dt\label{gamma1}\\
\gamma_{2}^{jl}(x)  &  =\int_{0}^{1}\frac{\partial F^{jl}}{\partial x_{p}}%
(\xi_{0}+t\vec{V})dt. \label{gamma2}%
\end{align}

Letting $f=u^{h_{p}}$ and
\begin{align}
\psi^{k}\left(  x\right)   &  =a^{k}(x,Du,D^{2}u)\label{gamma3}\\
\zeta(x)  &  =b(x,Du,D^{2}u), \label{gamma4}%
\end{align}
we arrive the following equation by plugging the above expressions into
(\ref{diff1}) governing the difference quotients
\[
\int_{B_{1}}\left(  \beta^{ij,kl}f_{ik}\eta_{jl}+\gamma_{1}^{jl,k}f_{k}%
\eta_{jl}+\gamma_{2}^{jl}\eta_{jl}+\psi^{k}\eta_{k}^{-h_{p}}+\zeta\eta
^{-h_{p}}\right)  dx=0.
\]
This linearized equation, which holds true provided $\eta\in C_{c}^{\infty
}(B_{1-h})$ governs difference quotients for solutions to (\ref{main}).
\ Further simplifying notation we define
\begin{equation}
\gamma^{jl}(x)=\int_{0}^{1}\left(  \frac{\partial F^{jl}}{\partial u_{k}}%
({\xi_{0}+t\vec{V}})f_{k}+\frac{\partial F^{jl}}{\partial x_{p}}(\xi_{0}%
+t\vec{V})\right)  dt \label{f}%
\end{equation}
to get
\begin{equation}
\int_{B_{1}}\left(  \beta^{ij,kl}f_{ik}\eta_{jl}+\gamma^{jl}\eta_{jl}+\psi
^{k}\eta_{k}^{-h_{p}}+\zeta\eta^{-h_{p}}\right)  dx=0. \label{main3}%
\end{equation}

\medskip Observe that since we do not start with a continuous Hessian, we
leave the expressions for the above leading coefficients in their integral form.

\begin{definition}
We define the nonlinear fourth order equation (\ref{main}) to be $\Lambda
$-\textbf{uniform} on a convex neighborhood $U\subset\mathbb{R}^{n}%
\times\mathbb{R}^{n}\times S^{n\times n}$ if the standard Legendre ellipticity
condition is satisfied for any $\xi\in U $%

\begin{equation}
\frac{\partial F^{jl}}{\partial u_{ik}}(\xi)\sigma_{ij}\sigma_{kl}\geq
\Lambda\left\Vert \sigma\right\Vert ^{2},\text{ $\forall$ }\sigma\text{ $\in
S^{n\times n}$}. \label{Bcondition_1}%
\end{equation}

\end{definition}

\medskip

\begin{remark}
\label{same_lambda}While this definition is tailored to equations of the form
(\ref{main}) it is important to note that it also applies to linear equations
of the form (\ref{main3}), in which case
\[
F^{jl}(x)=\beta^{ij,kl}(x)f_{ik}+\gamma^{jl}(x)
\]
and
\[
\frac{\partial F^{jl}}{\partial u_{ik}}=\beta^{ij,kl}(x).
\]
Thus when the nonlinear equation
(\ref{main}) is $\Lambda$-uniform, then so is the linearized equation
(\ref{main3}).
\end{remark}

\medskip We will use the following results to prove higher regularity in
section \ref{regresults}. We state the results here for the convenience of
the reader.

\begin{theorem}
\label{five}\cite[Theorem 2.1]{BW1}. Suppose $w\in W^{2,2}(B_{r})$ satisfies
the $\Lambda$-uniform constant coefficient equation%
\begin{align*}
\int c_{0}^{ik,jl}w_{ik}\eta_{jl}dx =0, \ \ \ \forall\eta\in C_{0}^{\infty
}(B_{r}). \label{ccoef}%
\end{align*}
Then for any $0<\rho\leq r$ there holds%

\begin{align*}
\int_{B_{\rho}}|D^{2}w|^{2}  &  \leq C_{1}\left(  \frac{\rho}{r}\right)
^{n}||D^{2}w||_{L^{2}(B_{r})}^{2},\\
\int_{B_{\rho}}|D^{2}w-(D^{2}w)_{\rho}|^{2}  &  \leq C_{2}\left(  \frac{\rho
}{r}\right)  ^{n+2}\int_{B_{r}}|D^{2}w-(D^{2}w)_{r}|^{2}%
\end{align*}
where $C_{1},C_{2}$ depend on the ellipticity constant $\Lambda$ and
$(D^{2}w)_{\rho}$ is the average value\ of $D^{2}w$ on a ball of radius $\rho$.
\end{theorem}

\medskip

\begin{corollary}
\label{Cor2} \cite[Corollary 2.2]{BW1}. \textit{Suppose }$w$\textit{ is as in
the Theorem \ref{five}. Then for any~}$u\in W^{2,2}(B_{r}),$ and\textit{ for
any~} $0<\rho\leq r,$ there holds%
\[
\int_{B_{\rho}}\left\vert D^{2}u\right\vert ^{2}\leq4C_{1}\left(  \frac{\rho
}{r}\right)  ^{n}\left\Vert D^{2}u\right\Vert _{L^{2}(B_{r})}^{2}+\left(
2+8C_{1}\right)  \left\Vert D^{2}(w-u)\right\Vert _{L^{2}(B_{r})}^{2}
\label{twothree}%
\]
and
\begin{align*}
\int_{B_{\rho}}\left\vert D^{2}u-(D^{2}u)_{\rho}\right\vert ^{2} \leq
4C_{2}\left(  \frac{\rho}{r}\right)  ^{n+2}\int_{B_{r}}\left\vert
D^{2}u-(D^{2}u)_{r}\right\vert ^{2} +\left(  8+16C_{2}\right)  \int_{B_{r}%
}\left\vert D^{2}(u-w)\right\vert ^{2} \label{twofive}%
\end{align*}
where $C_{1}, C_{2}$ depend on the ellipticity constant $\Lambda.$

\end{corollary}

\medskip

\begin{lemma}
\label{HanLin} \cite[Lemma 3.4]{HanLin}. Let $\phi$ be a nonnegative and
nondecreasing function on $[0,R].$ \ Suppose that
\[
\phi(\rho)\leq A\left[  \left(  \frac{\rho}{r}\right)  ^{\alpha}%
+\varepsilon\right]  \phi(r)+Br^{\beta}%
\]
for any $0<\rho\leq r\leq R,$ with $A,B,\alpha,\beta$ nonnegative constants
and $\beta<\alpha.$ \ Then for any $\gamma\in(\beta,\alpha),$ there exists a
constant $\varepsilon^{\ast}=\varepsilon^{\ast}(A,\alpha,\beta,\gamma)$ such
that if $\varepsilon<\varepsilon^{\ast}$ we have for all $0<\rho\leq r\leq R$%
\[
\phi(\rho)\leq c\left[  \left(  \frac{\rho}{r}\right)  ^{\gamma}%
\phi(r)+Br^{\beta}\right]
\]
where $c$ is a positive constant depending on $A,\alpha,\beta,\gamma.$ \ In
particular, we have for any $0<r\leq R$%
\[
\phi(r)\leq c\left[  \frac{\phi(R)}{R^{\gamma}}r^{\gamma}+Br^{\beta}\right]
.
\]

\end{lemma}

The following boundary value problem existence result should come as no
surprise, but is included for completeness.

\begin{lemma}
\label{LM} Suppose that $g\in W^{2,2}(B_{r})$, and $c_{0}^{ij,kl}$ is as in
Theorem \ref{five}. There exists a unique solution $w\in W^{2,2}(B_{r})$
solving the following BVP
\begin{align*}
\int_{B_{r}}c_{0}^{ij,kl}w_{ik}\eta_{jl}dx  &  =0, \ \ \ \text{ $\forall$}\eta\in
C_{0}^{\infty}(B_{r})\\
w  &  =g,\ \ D w=D g\ \ \ \ \ \text{on }\partial B_{r}(y).
\end{align*}

\end{lemma}

\begin{proof}
By \cite[Corollary 6.48, 6.49]{Folland} the boundary condition is equivalent to $w-g \in  H_{0}^{2}(B_{r}) $. The problem will be solved if we can find a function $v = w-g \in H_{0}^{2}(B_{r})$ such that 
\[
\int_{B_{r}}c_{0}^{ij,kl}\left(  w-g\right)  _{ik}\eta_{jl}dx+\int_{B_{r}%
}c_{0}^{ij,kl}g_{ik}\eta_{jl}dx=0.
\]
So it suffices to solve the problem 
\begin{align*}
\int_{B_{r}}c_{0}^{ij,kl}v_{ik}\eta_{jl}dx  &  =-\int_{B_{r}}c_{0}%
^{ij,kl}g_{ik}\eta_{jl}dx\\
 v &\in  H_{0}^{2}(B_{r}). 
\end{align*}
First, we claim that
\begin{equation}
\langle\phi,\varphi\rangle=\int_{B_r} c_{0}^{ij,kl}\phi_{ik}\varphi_{jl}dx
\label{hilbert}%
\end{equation}
defines a Hilbert space norm on the function space $H_{0}^{2}(B_{r}).$ In
other words, the norm defined by (\ref{hilbert}) is equivalent to the
$W^{2,2}_0(B_{r})$ norm
and the inner product is symmetric. First note that by the Legendre condition
\[
\langle\phi,\phi\rangle\geq\Lambda_{1}\int_{B_r}\left\vert D^{2}\phi\right\vert ^{2}
\]
where $\Lambda_1$ depends on $\Lambda, n$, and because $c^{ij,kl}_{0}$ is bounded we have
\[
\langle\phi,\phi\rangle\leq\Lambda_{2}\int_{B_r}\left\vert D^{2}\phi\right\vert^{2}
\]
where $\Lambda_2$ depends on $n,\|c_0^{ij,kl}\|_{L^\infty}$ for $1\leq i,j,k,l,\leq n$. 
Using the Poincar\'e inequality \cite[(7.44)]{GT},
for any $\phi\in W_{0}^{2,2}$ (hence $D\phi\in W^{1,2}_{0}$)
\[
\frac{1}{C}\langle\phi,\phi\rangle\leq\left\Vert \phi\right\Vert
_{W^{2,2}(B_{r})}^{2}\leq C\langle\phi,\phi\rangle.
\]
Thus the norm $\langle\phi,\phi\rangle$ is continuous with respect to the
$W^{2,2}$ norm.

Next we argue symmetry of \eqref{hilbert}: For $\phi,\varphi\in H_{0}%
^{2}(B_{r})$ we may take $\phi_{m},\varphi_{m}\in C_{c}^{\infty}(B_{r})\cap
W^{2,2}(B_{r}),$ which converge respectively to $\phi,\varphi$ in
$W^{2,2\text{ }},$ as $m\rightarrow\infty.$ We have
\begin{align*}
\langle\phi,\varphi\rangle &  =\lim_{m\rightarrow\infty}\langle\phi
_{m},\varphi_{m}\rangle\\
&  =\lim_{m\rightarrow\infty}\int_{B_r} c_{0}^{ij,kl}\left(  \phi_{m}\right)
_{ik}\left(  \varphi_{m}\right)  _{jl}dx\\
&  =(-1)^{2}\lim_{m\rightarrow\infty}\int_{B_r} c_{0}^{ij,kl}\left(  \phi
_{m}\right)  _{ikjl}\left(  \varphi_{m}\right)  dx\\
&  =(-1)^{4}\lim_{m\rightarrow\infty}\int_{B_r} c_{0}^{ij,kl}\left(  \phi
_{m}\right)  _{jl}\left(  \varphi_{m}\right)  _{ik}dx\\
&  =\lim_{m\rightarrow\infty}\langle\varphi_{m},\phi_{m}\rangle\\
&  =\langle\varphi,\phi\rangle.
\end{align*}

The linear operator
\[
f(\phi)=-\int_{B_{r}}c_{0}^{ij,kl}g_{ik}\phi_{jl}dx
\]
on $W^{2,2}_{0}(B_{r})$ is bounded with respect to the norm defined by
(\ref{hilbert}). To see this, take any $\phi$ in $H_{0}^{2}(B_{r})$, then
\begin{align*}
f(\phi)  &  =-\int_{B_{r}}c_{0}^{ij,kl}g_{ik}\phi_{jl}dx\\
&\leq C_{1}\left\Vert
g\right\Vert _{W^{2,2}(B_{r})}\left\Vert \phi\right\Vert _{W^{2,2}(B_{r})}\\
&  \leq C_{1}\left\Vert g\right\Vert _{W^{2,2}(B_{r})}C_{2}\left(  \langle
\phi,\phi\rangle\right)  ^{1/2}.
\end{align*}
By the Riesz representation theorem, there is a unique solution $v\in
H_{0}^{2}(B_{r})$ such that
\[
f(\eta)=\langle\eta,v\rangle=\int_{B_r} c_{0}^{ij,kl}v_{ik}\eta_{jl}dx
\]
that is
\[
-\int_{B_{r}}c_{0}^{ij,kl}g_{ik}\eta_{jl}dx=\int_{B_r} c_{0}^{ij,kl}v_{ik}\eta
_{jl}dx.
\]

Thus we
can let
\[
w=v+g.
\]
This gives the solvability of the boundary value problem in $H^2_0(B_r)$. 
\end{proof}

\subsection{Main regularity results\label{regresults}}

We will establish Theorem \ref{main1:Intro} by first proving the solution is
$C^{2,\alpha}$ and then by bootstrapping for smoothness. We state our two main
regularity boosting results below.


\begin{theorem}
\label{main_1} Suppose that $u\in W^{2,\infty}(B_{1})$ is a weak solution of
the $\Lambda$-uniform equation (\ref{main}) on $B_{1}$, such that
\[
\left\{  (x,Du(x),D^{2}u(x)):x\in B_{1}\right\}  \subset U.
\]
Fix $\alpha\in(0,1)$ and let $q=\frac{n}{2(1-\alpha)}.$ \ There
exists an $\varepsilon_{0}>0$, depending only on $\Lambda, \alpha$ and $n$
such that if the coefficients $\beta^{ij,kl}$ given by (\ref{bterm}) satisfy
\begin{equation}
\label{cons}\left|  \beta^{ij,kl}(x,Du,D^{2}u)-a_{0}^{ij,kl}\right|
<\varepsilon_{0}%
\end{equation}
where $a_{0}^{ij,kl}=\frac{\partial F^{jl}}{\partial u_{ik}}\left(
\xi\right)  $ for some $\xi\in U,$ then $u\in C^{2,\alpha}(B_{1})$ with
\[
||D^{2}u||_{C^{\alpha}(B_{1/4})}\leq C(\Lambda, \alpha,||u||_{W^{2,\infty
}(B_{1})}, \left\Vert DF \right\Vert _{L^{\infty}(U)},\left\Vert
a^{k}\right\Vert _{L^{\infty}(U)}, \left\Vert b \right\Vert _{L^{\infty}%
(U)}).
\]

\end{theorem}

\begin{theorem}
\label{main_2} Suppose that $u\in C^{2,\alpha}(B_{1})$ satisfies the $\Lambda
$-uniform equation (\ref{main}) on $B_{1}$. Then $u$ is smooth in $B_{1}$.
\end{theorem}

\begin{remark}
The closeness condition (\ref{cons}) is not needed to reach $W^{3,2}$ from
$W^{2,\infty}$. It is used to bootstrap to $C^{2,\alpha}$ from $W^{3,2}$, and
$C^{2,\alpha}$ is enough to bootstrap further.
\end{remark}

\subsection{Proof of Theorem \ref{main_1}}

To boost up regularity, we will work with equation \eqref{main3} on the
difference quotient $u^{h}_{p}$, rather than directly on \eqref{main} for $u$.
Given a solution $f$ to \eqref{main3}, we begin with bounding its $W^{2,2}$
norm in terms of its $W^{1,\infty}$ norm in Proposition \ref{reg3}, then in
Proposition \ref{reg}, we show that the $C^{1,\alpha}$ norm of $f$ depends on
its $W^{2,2}$ norm. This follows essentially the same arguments as in
\cite[Lemma 3.1]{CW} and \cite[Proposition 1.3]{BW1}.

Theorem \ref{main_1} will then follow from Propositions \ref{reg} and
\ref{reg3}, by taking $f=u^{h}_{p}$ therein.

\begin{proposition}
\label{reg3} Suppose that $f\in W^{2,\infty}(B_{1})$ satisfies the uniformly
elliptic weak double divergence equation (\ref{main3}) on $B_{1}$. Then $f$
satisfies the following estimate:
\begin{equation}
||f||_{W^{2,2}(B_{1/2})}\leq C\left(  \Lambda,\left\Vert f\right\Vert
_{W^{1,\infty}(B_{1})},\left\Vert \psi\right\Vert _{L^{2}(B_{1})},\left\Vert
\zeta\right\Vert _{L^{2}(B_{1})},\left\Vert \beta\right\Vert _{L^{\infty
}(B_{1})}\right)  .\text{ }%
\end{equation}

\end{proposition}

\begin{proof}
Assuming $f\in W^{2,\infty}(B_{1})$, $f$ will be $W^{2,2}$ and the function
$\tau^{4}f$ can be approximated by functions $\eta\in C_{c}^{\infty}(B_{3/4})$
in $W^{2,2}$ norm for $\tau$ smooth compactly supported on $B_{3/4}$ which is
$1$ on $B_{1/2}.$ Thus
\[
\int_{B_{1}}\bigg[
\begin{array}
[c]{c}%
\beta^{ij,kl}f_{ik}\left(  \tau^{4}f\right)  _{jl}+\gamma^{jl}\left(  \tau
^{4}f\right)  _{jl} +\psi^{k} \left(  \tau^{4}f\right)  _{k}^{-h_{p}}%
+\zeta\left(  \tau^{4}f\right)  ^{-h_{p}}%
\end{array}
\bigg]dx=0.
\]
Applying uniform ellipticity to the first term of the above expression, we
get
\begin{align}
\Lambda\int_{B_{1}}\tau^{4}\left\vert D^{2}f\right\vert ^{2}dx  &  \leq
\int_{B_{1}}\left\vert \beta^{ij,kl}f_{ik}\left(  \left(  \tau^{4}\right)
_{jl}f+\left(  \tau^{4}\right)  _{l}f_{j}+\left(  \tau^{4}\right)  _{j}%
f_{l}\right)  \right\vert \, dx\label{term3b}\\
&  +\int_{B_{1}}\left(  \left\vert \gamma^{jl}\left(  \tau^{4}f\right)
_{jl}\right\vert +\left\vert \psi^{k} \left(  \tau^{4}f\right)  _{k}^{-h_{p}%
}\right\vert +\left\vert \zeta\left(  \tau^{4}f\right)  ^{-h_{p}}\right\vert
\right)  dx.\nonumber
\end{align}
Straightforward use of inequalities gives%
\begin{align*}
\int_{B_{1}}  &  \left\vert \beta^{ij,kl}f_{ik}\left(  \left(  \tau
^{4}\right)  _{jl}f+\left(  \tau^{4}\right)  _{l}f_{j}+\left(  \tau
^{4}\right)  _{j}f_{l}\right)  \right\vert dx\\
&  \leq C\left(  D\tau,D^{2}\tau,\left\Vert f\right\Vert _{W^{1,\infty}%
},\left\Vert \beta\right\Vert _{L^{\infty}}\right)  \int_{B_{1}}\tau
^{2}\left\vert D^{2}f\right\vert dx\\
&  \leq C\left(  D\tau,D^{2}\tau,\left\Vert f\right\Vert _{W^{1,\infty}%
},\left\Vert \beta\right\Vert _{L^{\infty}}\right)  \left(  \frac
{1}{\varepsilon}+\varepsilon\int_{B_{1}}\tau^{4}\left\vert D^{2}f\right\vert
^{2}dx\right)  .
\end{align*}
Similarly%
\begin{equation}
\int_{B_{1}}\left\vert \gamma^{jl}\left(  \tau^{4}f\right)  _{jl}\right\vert
dx\leq C\left(  D\tau,D^{2}\tau,\left\Vert f\right\Vert _{W^{1,\infty}%
},\left\Vert \beta\right\Vert _{L^{\infty}}\right)  \left(  \frac
{1}{\varepsilon}+\varepsilon\int_{B_{1}}\tau^{4}\left\vert D^{2}f\right\vert
^{2}dx\right)  . \label{term2b}%
\end{equation}
Now for
\begin{equation}
\int_{B_{1}}\left\vert \psi^{k} \left(  \tau^{4}f\right)  _{k}^{-h_{p}%
}\right\vert dx \label{term222}%
\end{equation}
observe that
\begin{align*}
\int_{B_{1}}\left|  \psi^{k}\frac{\left(  \tau^{4}f\right)  _{k}%
(x-h_{p})-\left(  \tau^{4}f\right)  _{k}}{h}\right|  dx  &  =\int_{B_{1}%
}\left|  \psi^{k}\right|  \left|  \int_{0}^{1}D\left(  \tau^{4}f\right)
_{k}(x-th_{p})dt\right|  dx\\
&  \leq\int_{0}^{1}\int_{B_{1}}\left|  \psi^{k}\right|  \left|  D \left(
\tau^{4}f\right)  _{k} (x-th_{p})\right|  dxdt\\
&  \leq\int_{0}^{1}\left\Vert \psi\right\Vert _{L^{2}(B_{1})}\left\Vert
D^{2}\left(  \tau^{4}f\right)  \right\Vert _{L^{2}(B_{1})}dt\\
&  =\left\Vert \psi\right\Vert _{L^{2}(B_{1})}\left\Vert D^{2}\left(  \tau
^{4}f\right)  \right\Vert _{L^{2}(B_{1})}%
\end{align*}
which can be treated as in (\ref{term2b})%
\[
\int_{B_{1}}\left\vert \psi^{k} \left(  \tau^{4}f\right)  _{k}^{-h_{p}%
}\right\vert dx\leq C\left(  D^{2}\tau,\left\Vert f\right\Vert _{W^{1,\infty}%
},\left\Vert \psi\right\Vert _{L^{2}(B_{1})}\right)  \left(  \frac
{1}{\varepsilon}+\varepsilon\int_{B_{1}}\tau^{4}\left\vert D^{2}f\right\vert
^{2}dx\right)  .
\]
Finally, treating the last term in (\ref{term3b}) similarly as for
(\ref{term222}), we can bound (\ref{term3b}) in lower order terms of $f.$

Combining and using the appropriately chosen $\tau$, we choose $\varepsilon$
appropriately in the above equation and in (\ref{term2b}), to get
\[
\frac{\Lambda}{2}\int_{B_{1/2}}\left\vert D^{2}f\right\vert ^{2}dx\leq
C\left(  \left\Vert f\right\Vert _{W^{1,\infty}(B_{1})},\left\Vert
\psi\right\Vert _{L^{2}(B_{1})},\left\Vert \zeta\right\Vert _{L^{2}(B_{1}%
)},\left\Vert \beta\right\Vert _{L^{\infty}(B_{1})}\right)  ,
\]
therefore complete the proof.
\end{proof}

Our next result is key in achieving $C^{2,\alpha}$ regularity of $u$.

\begin{proposition}
\label{reg} For a fixed $h_{p}$ with $\left\vert h\right\vert <\frac{1}{100}$
suppose that $f\in W^{2,2}(B_{1})$ satisfies the uniformly elliptic double
divergence equation \eqref{main3} weakly on $B_{3/4}(0)$. Suppose that
$\gamma^{jl},\psi^{k}, \zeta\in L^{2q}$ with $q=\frac{n}{2-2\alpha}, \alpha
\in(0,1)$. Then, there is an $\varepsilon_{0}(n, \Lambda, \alpha)>0$, such
that if \eqref{cons} holds as in Theorem \ref{main_1} then we have $Df\in
C^{\alpha}(B_{1/4})$ and the estimates:
\begin{equation}
||Df||_{C^{\alpha}(B_{1/4})}\leq C(\Lambda,\alpha, ||f||_{W^{2,2}(B_{1/2}%
)},\left\Vert \gamma^{jl}\right\Vert _{L^{2q}(B_{1}),}\left\Vert \psi
^{k}\right\Vert _{L^{2q}(B_{1}) },\left\Vert \zeta\right\Vert _{L^{2q}(B_{1}%
)}). \label{est1}%
\end{equation}

\end{proposition}

\begin{proof}
Pick an arbitrary point $y\in$ $B_{1/4}$. Then $B_{r}(y)\subset B_{3/4}\text{
}$ for any fixed $r<1/2.$

We write $v=f-w,$ where $w$ satisfies the following constant coefficient
partial differential equation on ${B}_{r}(y)\subset B_{3/4}$:
\begin{align*}
\int_{B_{r}(y)}  &  a_{0}^{ij,kl}w_{ik}\eta_{jl}dx=0,\ \ \ \text{ $\forall$%
}\eta\in C_{0}^{\infty}(B_{r}(y))\\
w  &  =f,\ \ D w= D f\ \ \ \ \ \text{on }\partial B_{r}(y).
\end{align*}
Here $a_{0}^{ij,kl}$ is the symbol occurring in our assumption \eqref{cons}.
This solution exists by Lemma \ref{LM} and is smooth on the interior of
$B_{r}(y)$ \cite[Theorem 6.33]{Folland}.

We may extend $v$ to a function (still named $v$) on $B_{3/4}$ by defining
$v=0$ on $B_{3/4}\backslash B_{r}(y).$ \ As the original $v\in H_{0}^{2}%
(B_{r}(y))$ is the limit of $C_{c}^{\infty}(B_{r}(y))$ functions $\eta^{(m)}$
it follows that the extended $v$ must also remain in $H_{0}^{2}(B_{3/4}).$

Now because $v$ is the $W^{2,2}(B_{r}(y))$ limit of functions $\eta^{(m)}$
$\in C_{c}^{\infty}(B_{r}(y))\subset C_{c}^{\infty}(B_{3/4})$ we may also
write
\begin{align}
\int_{B_{r}(y)}a_{0}^{ij,kl}v_{ik}v_{jl}dx  &  =\lim_{m\rightarrow\infty}%
\int_{B_{r}(y)}a_{0}^{ij,kl}v_{ik}(\eta^{(m)})_{jl}dx\nonumber\\
&  =\lim_{m\rightarrow\infty}\int_{B_{r}(y)}a_{0}^{ij,kl}f_{ik}(\eta
^{(m)})_{jl}dx\nonumber\\
&  =\lim_{m\rightarrow\infty}\int_{B_{3/4}}a_{0}^{ij,kl}f_{ik}(\eta
^{(m)})_{jl}dx\nonumber\\
&  =\int_{B_{3/4}}a_{0}^{ij,kl}f_{ik}v_{jl}dx. \label{eq34}%
\end{align}
Now taking limits of (\ref{main3}) for $\eta^{(m)}\rightarrow v$ we conclude
that
\begin{equation}
\int_{B_{3/4}}\left(  \beta^{ij,kl}f_{ik}v_{jl}+\gamma^{jl}v_{jl}+\psi
^{k}v_{k}^{-h_{p}}+\zeta v^{-h_{p}}\right)  dx=0. \label{eq35}%
\end{equation}
Now we subtract (\ref{eq35}) from (\ref{eq34})
\begin{align}
\int_{B_{r}(y)}a_{0}^{ij,kl}v_{ik}v_{jl}dx  &  =\int_{B_{3/4}}a_{0}%
^{ij,kl}f_{ik}v_{jl}dx-\int_{B_{3/4}}\left(  \beta^{ij,kl}f_{ik}v_{jl}%
+\gamma^{jl}v_{jl}+\psi^{k}v_{k}^{-h_{p}}+\zeta v^{-h_{p}}\right)
dx\label{eq36}\\
&  =\int_{B_{3/4}}\left(  a_{0}^{ij,kl}-\beta^{ij,kl}\right)  f_{ik}%
v_{jl}dx-\int_{B_{3/4}}\gamma^{jl}v_{jl}dx-\int_{B_{3/4}}\left(  \psi^{k}%
v_{k}^{-h_{p}}+\zeta v^{-h_{p}}\right)  dx.\nonumber
\end{align}
First we note that our condition (\ref{cons}), for an $\varepsilon_{0}$ yet to
determined, gives us%
\begin{equation}
\int_{B_{3/4}}\left\vert (a_{0}^{ij,kl}-\beta^{ij,kl})f_{ik}v_{jl}\right\vert
dx\leq\varepsilon_{0}\left\Vert D^{2}f\right\Vert _{L^{2}(B_{r}(y))}\left\Vert
D^{2}v\right\Vert _{L^{2}(B_{r}(y))}, \label{eq37}%
\end{equation}
making use of the fact that $v$ is supported in $B_{r}(y).$ Next, by
H\"{o}lder's inequality
\begin{equation}
\int_{B_{3/4}}|\gamma^{jl}v_{jl}|dx\leq C(n)\left\Vert \gamma\right\Vert
_{L^{2}(B_{r}(y))}\left\Vert D^{2}v\right\Vert _{L^{2}(B_{r}(y))}\leq
C(n)\left\Vert \gamma\right\Vert _{L^{2q}(B_{r}(y))}r^{\frac{n-2+2\alpha}{2}%
}\left\Vert D^{2}v\right\Vert _{L^{2}(B_{r}(y))} \label{eq38}%
\end{equation}
where $q=\frac{n}{2(1-\alpha)}$.

For the third term
\begin{align}
\int_{B_{3/4}}  &  \left\vert \psi^{k}\frac{v_{k}(x-h_{p})-v_{k}(x)}%
{h}\right\vert dx=\lim_{m\rightarrow\infty}\int_{B_{3/4}}\left\vert \psi
^{k}\frac{\left(  \eta^{(m)}\right)  _{k}(x-h_{p})-\left(  \eta^{(m)}\right)
_{k}(x)}{h}\right\vert dx\nonumber\\
&  =\lim_{m\rightarrow\infty}\int_{B_{3/4}}\left\vert \psi^{k}\int_{0}%
^{1}\left(  -D_{pk}\eta^{(m)}(x-th_{p})\right)  dt\right\vert dx\nonumber\\
&  \leq\lim_{m\rightarrow\infty}\int_{B_{3/4}}\left\vert \psi^{k}\right\vert
\int_{0}^{1}\left\vert D_{pk}\eta^{(m)}(x-th_{p})dt\right\vert dx\nonumber\\
&  \leq\lim_{m\rightarrow\infty}\int_{0}^{1}\int_{B_{3/4}}\left\vert \psi
^{k}\right\vert \left\vert D_{pk}\eta^{(m)}(x-th_{p})\right\vert
dxdt\ \ \ \ \ (\mbox{Tonelli's Theorem})\nonumber\\
&  \leq\int_{0}^{1}\int_{B_{3/4}}\left\vert \psi^{k}\right\vert \left\vert
D^{2}v(x-th_{p})\right\vert dxdt\ \ \ \ \ (\mbox{Fatou's Lemma})\nonumber\\
&  =\int_{0}^{1}\int_{B_{r+h}(y)}\left\vert \psi^{k}\right\vert \left\vert
D^{2}v(x-th_{p})\right\vert
dxdt\ \ \ \ \ (\mbox{supp $v\subset B_r(y)$})\nonumber\\
&  \leq\left\Vert \psi\right\Vert _{L^{2}(B_{r+h}(y))}\left\Vert
D^{2}v\right\Vert _{L^{2}(B_{r}(y))}%
\ \ \ \ \ \ (\mbox{Cauchy-Schwarz inequality})\nonumber\\
&  \leq C(n)\left\Vert \psi\right\Vert _{L^{2q}(B_{r+h}(y))}r^{\frac
{n-2+2\alpha}{2}}\left\Vert D^{2}v\right\Vert _{L^{2}(B_{r}(y))}%
.\ \ \ (\mbox{H\"older's inequality}) \label{eq39}%
\end{align}

A\ similar computation yields
\begin{align}
\int_{B_{3/4}}\left\vert \zeta(x)v^{-h_{p}}(x)dx\right\vert  &  \leq\Vert
\zeta\Vert_{L^{2}(B_{r+h}(y))}\cdot\Vert Dv\Vert_{L^{2}(B_{r}(y))}\nonumber\\
&  \leq C(n)\Vert\zeta\Vert_{L^{2q}(B_{r+h}(y))}r^{\frac{n-2+2\alpha}{2}}\cdot
C_{p}|B_{r}(y)|^{\frac{1}{n}}\left\Vert D^{2}v\right\Vert _{L^{2}(B_{r}(y))}
\label{e310}%
\end{align}
where $C_{p}$ is from the Poincar\'{e} inequality \cite[(7.44)]{GT}.

\medskip

Now since $a_{0}^{ij,kl}$ has an ellipticity constant $\Lambda$, plugging the
bounds (\ref{eq37}), (\ref{eq38}), (\ref{eq39}), (\ref{e310}) into
(\ref{eq36}), we have (collecting dimensional constants into a new $C(n)$)
\begin{align*}
\Lambda\left\Vert D^{2}v\right\Vert _{L^{2}(B_{r}(y))}^{2}  &  \leq
\varepsilon_{0}\left\Vert D^{2}f\right\Vert _{L^{2}(B_{r}(y))}\left\Vert
D^{2}v\right\Vert _{L^{2}(B_{r}(y))}+C(n)\left\Vert \gamma\right\Vert
_{L^{2q}}r^{\frac{n-2+2\alpha}{2}}\left\Vert D^{2}v\right\Vert _{L^{2}%
(B_{r}(y))}\\
&  +C(n)\left\Vert \psi\right\Vert _{L^{2q}}r^{\frac{n-2+2\alpha}{2}%
}\left\Vert D^{2}v\right\Vert _{L^{2}(B_{r}(y))}+C(n)\left\Vert \zeta
\right\Vert _{L^{2q}}r^{\frac{n-2+2\alpha}{2}}\left\Vert D^{2}v\right\Vert
_{L^{2}(B_{r}(y))}.
\end{align*}
Dividing by $\left\Vert D^{2}v\right\Vert _{L^{2}(B_{r}(y))}$ and collecting%
\[
\Lambda\left\Vert D^{2}v\right\Vert _{L^{2}(B_{r}(y))}\leq\varepsilon
_{0}\left\Vert D^{2}f\right\Vert _{L^{2}(B_{r}(y))}+C(n)\left(  \left\Vert
\gamma\right\Vert _{L^{2q}}+\left\Vert \psi\right\Vert _{L^{2q}}+\left\Vert
\zeta\right\Vert _{L^{2q}}\right)  r^{\frac{n-2+2\alpha}{2}}.
\]
That is%
\[
\Lambda^{2}\left\Vert D^{2}v\right\Vert _{L^{2}(B_{r}(y))}^{2}\leq
2\varepsilon_{0}^{2}\left\Vert D^{2}f\right\Vert _{L^{2}(B_{r}(y))}%
^{2}+Kr^{n-2+2\alpha}%
\]
for (again modifying $C(n)$)
\[
K=C(n)\left(  \left\Vert \gamma\right\Vert _{L^{2q}}^{2}+\left\Vert
\psi\right\Vert _{L^{2q}}^{2}+\left\Vert \zeta\right\Vert _{L^{2q}}%
^{2}\right)  .
\]

Recalling $f=v+w$ and Corollary \ref{Cor2}
\[
\int_{B_{\rho}(y)}\left\vert D^{2}f\right\vert ^{2}\leq4C_{1}\left(
\frac{\rho}{r}\right)  ^{n}\left\Vert D^{2}f\right\Vert _{L^{2}(B_{r}(y))}%
^{2}+\left(  2+8C_{1}\right)  \left\Vert D^{2}v\right\Vert _{L^{2}(B_{r}%
(y))}^{2}%
\]
for $C_{1}$ depending on the ellipticity of $a_{0}^{ij,kl}$ we see
\begin{equation}
\int_{B_{\rho}(y)}\left\vert D^{2}f\right\vert ^{2}\leq4C_{1}\left(
\frac{\rho}{r}\right)  ^{n}\left\Vert D^{2}f\right\Vert _{L^{2}(B_{r}(y))}%
^{2}+\frac{2\left(  2+8C_{1}\right)  }{\Lambda^{2}}\left(  \varepsilon_{0}^{2}
\left\Vert D^{2}f\right\Vert _{L^{2}(B_{r}(y))}^{2}+Kr^{n-2+2\alpha}\right)  .
\label{2hl}%
\end{equation}
Now, we would like to apply Lemma \ref{HanLin}. To this end, let
\begin{align*}
\phi(\rho)  &  =\int_{B_{\rho}}\left\vert D^{2}f\right\vert ^{2}\\
A  &  =4C_{1}\\
\varepsilon &  =\frac{2\left(  2+8C_{1}\right)  }{\Lambda^{2}}\varepsilon
_{0}^{2}\\
B  &  =\frac{2\left(  2+8C_{1}\right)  }{\Lambda^{2}}K\\
\alpha &  =n\\
\beta &  =n-2+2\alpha\\
\gamma &  =n-1\\
R  &  =\frac{1}{2}.
\end{align*}
To be clear, in order to avoid notational double-dipping, the notations
appearing on the left hand side of expressions above refer to constants as
they are named in Lemma \ref{HanLin}, while the right hand side refers to
constants as they appear previously in this proof so far. We observe that
(\ref{2hl}) can be written using notation on the left side of the above table
as
\begin{equation}
\phi(\rho)\leq A\left[  \left(  \frac{\rho}{r}\right)  ^{\alpha}%
+\varepsilon\right]  \phi(r)+Br^{\beta} \label{this}%
\end{equation}
for all $0<\rho\leq r<\frac{1}{2}.$ There exists a constant $\varepsilon
^{\ast}\left(  A,\alpha,\beta,\gamma\right)  $ so that (\ref{this}) allows us
to conclude that there is a constant $C>0$ such that
\[
\phi(\rho)\leq C\left[  \left(  \frac{\rho}{r}\right)  ^{n-1}\phi
(r)+Br^{n-2+2\alpha}\right]
\]
whenever
\begin{equation}
\frac{2\left(  2+8C_{1}\right)  }{\Lambda^{2}}\varepsilon_{0}^{2}%
\leq\varepsilon^{\ast}\left(  A,\alpha,\beta,\gamma\right)  .
\label{eps_0_def}%
\end{equation}
We pick one such $\varepsilon_{0}$. Thus
\begin{align*}
\phi(r)  &  \leq C\left[  2^{n-1}r^{n-1}\phi(\frac{1}{2})+Br^{n-2+2\alpha
}\right] \\
&  \leq C^{\prime}r^{n-2+2\alpha}%
\end{align*}
where $C^{\prime}$ depends on $\ \int_{B_{1/2}}\left\vert D^{2}f\right\vert
^{2},\Lambda,n,\alpha,$ and $\frac{2\left(  2+8C_{1}\right)  }{\Lambda^{2}}K$.

We now have that
\[
\int_{B_{r}}\left\vert D^{2}f\right\vert ^{2}\leq C^{\prime}r^{n-2+2\alpha}.
\]
Noting that we chose an arbitrary point in $B_{1/4}(0)$ we may apply Morrey's
Lemma \cite[Lemma 3, page 8]{SimonETH} to $Df$ to get the desired conclusion.
\end{proof}

\noindent\textit{{Proof of Theorem \ref{main_1}.}} Applying Proposition
\ref{reg3} we see that $u\in W^{3,2}$, with estimates controlled by
$\left\Vert u\right\Vert _{W^{2,\infty}}.$ The difference quotient
$f=u^{h}_{p}$ satisfies \eqref{main3} where now $f\in W^{2,2}$ with estimates.
Using the supremum norms of $DF, a^{k}, b$ and that $u\in W^{2,\infty}$, the
conditions on $\gamma^{jl},\psi^{k},\zeta$ in Proposition \ref{reg} are
fulfilled, namely, they are in $L^{2q}$. In light of Proposition \ref{reg} we
conclude $u_{p}^{h}\in C^{1,\alpha}$ with the estimate \eqref{est1} where we
note that now
\begin{align*}
\|f\|_{W^{1,\infty}}  &  =\left\|  \frac{u(x)-u(x-h_{p})}{h}\right\|
_{W^{1,\infty}}\\
&  =\mbox{ess sup} \left(  \left|  \frac{u(x)-u(x-h_{p})}{h}\right|  +\left|
\frac{Du(x)-Du(x-h_{p})}{h}\right|  \right) \\
&  \leq\mbox{Lip} (u) + \mbox{Lip}(Du)\\
&  \leq\mbox{ess sup}\left(  |u|+|Du|+|D^{2}u| \right) \\
&  = \|u\|_{W^{2,\infty}}%
\end{align*}
Letting $h\rightarrow0$ in \eqref{est1} yields the estimate that holds on $B_{1/4}.$   Now take any interior point $x_0$ and
consider the equation%
\begin{equation}
\partial_{y_{l}}\partial_{y_{j}}\tilde{F}^{jl}(y,Dv,D^{2}v)=\partial_{y_{k}%
}\tilde{a}^{k}(y,Dv,D^{2}v)-\tilde{b}(y,Dv,D^{2}v)\label{eqnv}%
\end{equation}
with
\begin{align*}
\tilde{F}^{jl}(y,Dv,D^{2}v)  & =F^{jl}(x_{0}+ry,rDv(x_{0}+ry),D^{2}%
v(x_{0}+ry))\\
\tilde{a}^{k}(y,Dv,D^{2}v)  & =ra^{k}(x_{0}+ry,rDv(x_{0}+ry),D^{2}%
v(x_{0}+ry)\\
\tilde{b}(x,Dv,D^{2}v)  & =r^{2}b(x_{0}+ry,rDv(x_{0}+ry),D^{2}v(x_{0}+ry).
\end{align*}
Suppose that
\[
B_{r}(x_{0})\subset B_{1}.
\]
Define
\[
v(y)=\frac{u(x_{0}+ry)}{r^{2}}.
\]
One can check that $v$ satisfies (\ref{eqnv}) on $B_{1}$ whenever $u$
satisfies (\ref{main}).

Noting that 
\[
\frac{\partial \tilde{F}^{jl}}{\partial v_{ik}}(y,Dv,D^{2}v)=\frac{\partial F^{jl}%
}{\partial u_{ik}}(x_{0}+ry,rDv(x_{0}+ry),D^{2}v(x_{0}+ry))
\]
we see equation (\ref{eqnv}) and the solution $v$ will  satisfy the closeness condition (\ref{cons}) as well.  This rescaling argument allows us to claim an estimate holds at any interior point in $B_{1}.$    \hfill $\square$

\subsection{Proof of Theorem \ref{main_2}}

We start by boosting regularity from $C^{2,\alpha}$ to $C^{3,\alpha}$.

\begin{proposition}
\label{reg1} Suppose that $u\in C^{2,\alpha}(B_{1})$ satisfies the $\Lambda
$-uniform equation (\ref{main}) on $B_{1}$, and let $0<\delta< \alpha$. Then
$D^{3}u\in C^{\alpha-\delta/2}(B_{1/5})$ and satisfies the following
estimate:
\begin{equation}
||D^{3}u||_{C^{\alpha-\delta/2}(B_{1/5})}\leq C(||u||_{W^{2,\infty}(B_{1}%
)},\Lambda,\alpha,\delta). \label{er}%
\end{equation}

\end{proposition}

\begin{proof}
We assume that $u$ enjoys uniform $C^{2,\alpha}$ estimates on $B_{9/10}.$ As
before we take a difference quotient of the solution $u$ to (\ref{main}) to
get (\ref{main3}) with $f=u^{h_{p}}$, for some $h<1/100$. Since $D^{2}u\in
C^{\alpha}(\bar{B}_{9/10})$, the measurable coefficients are now integrals of
H\"{o}lder continuous functions, when defined for any $x\in B_{3/4}$ as
follows:
\begin{align*}
\beta^{ij,kl}(x)  &  =\int_{0}^{1}\frac{\partial F^{jl}}{\partial u_{ik}}%
(\xi_{0}+t\vec{V})dt\in C^{\alpha}(B_{3/4})\\
\gamma^{jl}(x)  &  =\int_{0}^{1}\left(  \frac{\partial F^{jl}}{\partial u_{k}%
}({\xi_{0}+t\vec{V}})u_{k}^{h_{p}}+\frac{\partial F^{jl}}{\partial x_{p}}%
(\xi_{0}+t\vec{V})\right)  dt\in C^{\alpha}(B_{3/4}).
\end{align*}
Note also that $\psi^{k}(x)\in C^{\alpha}(B_{3/4}).$ \ In particular \
\[
\left\vert \beta^{ij,kl}(x)-\beta^{ij,kl}(y)\right\vert \leq C_{3}\left\vert
x-y\right\vert ^{\alpha}\text{ .}%
\]
Again, fixing $y\in B_{1/4}$ for a fixed $r<\frac{1}{2}$\ we let $w$ solve the
boundary value problem
\begin{align*}
\int_{B_{r}(y)}\beta^{ij,kl}(0)w_{ij}\eta_{kl}\,dx  &  =0,\ \ \ \forall\eta\in
C_{0}^{\infty}(B_{r}(y))\\
w  &  =f\text{, }D w=D f\ \ \ \text{ on }\partial B_{r}(y)
\end{align*}
and repeat verbatim the steps leading to (\ref{eq36}), with $a_{0}^{ij,kl}$
being replaced by $\beta^{ij,kl}(0),$ again taking $v=f-w\in H_{0}^{2}%
(B_{r}(y))$. Thus by (\ref{main3})
\[
\int_{B_{r}(y)}\beta^{ij,kl}(0)v_{ij}v_{kl}dx=\int_{B_{r}(y)}\left(
\beta^{ij,kl}(0)-\beta^{ij,kl}(x)\right)  f_{ik}v_{jl}dx-\int_{B_{r}%
(y)}\left(  \gamma^{jl}v_{jl} + \psi^{k}v_{k}^{-h_{p}}+\zeta v^{-h_{p}%
}\right)  dx.
\]
Now this time, we define
\begin{equation}
\Upsilon(r)=\sup\left\{  \left|  \beta^{ij,kl}(x)-\beta^{ij,kl}(x^{\prime})
\right|  \mid x,x^{\prime}\in B_{r}(y)\right\}  \label{nbb}%
\end{equation}
which enjoys an estimate from the H\"{o}lder estimate on $D^{2}u:$
\begin{equation}
\Upsilon(r)\leq C_{4}r^{\alpha}. \label{nbb1}%
\end{equation}

Since $v\in H_{0}^{2}(B_{r}(y))$, we have, via integration by parts, that
\begin{align*}
\int_{B_{1}}\gamma^{jl}(y)v_{jl}(x)\,dx  &  =0\\
\int_{B_{1}}\psi^{k}(y)v_{k}^{-h_{p}}(x)\,dx  &  =0
\end{align*}
and
\[
\int_{B_{1}}\zeta(y)v^{-h_{p}}(x)dx =\zeta(y)\frac{1}{h}\left(  \int_{B_{1}%
}v(x-h_{p})dx-\int_{B_{1}}v(x)dx\right)  =0
\]
so we may write%
\begin{align*}
\int_{B_{1}}  &  \left(  \gamma^{jl}v_{jl}+\psi^{k}v_{k}^{-h_{p}}+\zeta
v^{-h_{p}}\right)  dx\\
&  =\int_{B_{1}}\left(  \left[  \gamma^{jl}(x)-\gamma^{jl}(y)\right]
v_{jl}+\left[  \psi^{k}(x)-\psi^{k}(y)\right]  v_{k}^{-h_{p}}+\left[
\zeta(x)-\zeta(y)\right]  v^{-h_{p}}\right)  dx.
\end{align*}
Now
\[
\int_{B_{1}}\left\vert [\gamma^{jl}(x)-\gamma^{jl}(y)]v_{jl}\right\vert
dx\leq\left\Vert \gamma(x)-\gamma(y)\right\Vert _{L^{2}(B_{r}(y))}\left\Vert
D^{2}v\right\Vert _{L^{2}(B_{r}(y))}\leq C_{5}\left(  r^{2\alpha}r^{n}\right)
^{\frac{1}{2}}\left\Vert D^{2}v\right\Vert _{L^{2}(B_{r})}%
\]
and similarly,%
\[
\int_{B_{1}}\left\vert [\psi^{k}(x)-\psi^{k}(y)]v_{k}^{-h_{p}}\right\vert
dx\leq C_{6}\left( r^{2\alpha}r^{n}\right)^{\frac{1}{2}}\left\Vert D^{2}%
v\right\Vert _{L^{2}(B_{r}(y))}%
\]

%

\begin{align}
\int_{B_{1/2}}\left\vert [\zeta(x)-\zeta(y)]v^{-h_{p}}(x)\right\vert dx &\leq
C_{7}\left( r^{2\alpha}r^{n}\right)^{\frac{1}{2}}\Vert Dv\Vert_{L^{2}(B_{r}(y)
)}\nonumber\\
&\leq C_{7}\left( r^{2\alpha}r^{n}\right)^{\frac{1}{2}} C_{p}|B_{r}|^{\frac{1}{n}%
}\left\Vert D^{2}v\right\Vert _{L^{2}(B_{r}(y))}\nonumber\\
&\leq C^{\prime}_{p}C_{7}\left(  r^{2\alpha}r^{n}\right)^{\frac{1}{2}}\left\Vert
D^{2}v\right\Vert _{L^{2}(B_{r}(y))}\nonumber
\end{align}
where $C_{p}$ is from the Poincar\'{e} inequality \cite[(7.44)]{GT},
$C_{p}^{\prime}=C_{p}|B_{1}|$, and

%

\begin{align*}
\left\vert \gamma(x)-\gamma(y)\right\vert  &  \leq C_{5}r^{\alpha}\\
\left\vert \psi(x)-\psi(y)\right\vert  &  \leq C_{6}r^{\alpha}\\
\left\vert \zeta(x)-\zeta(y)\right\vert  &  \leq C_{7}r^{\alpha}.
\end{align*}
(Recall the components of these functions are smooth as functions of $D^{2}u$
so these will be H\"{o}lder continuous now as $D^{2}u$ is H\"{o}lder
continuous.)\ \ Note that, for $\Lambda$ the ellipticity constant for $\beta$
we have
\[
\Lambda\left\Vert D^{2}v\right\Vert _{L^{2}(B_{r}(y))}^{2}\leq\Upsilon
(r)\left\Vert D^{2}f\right\Vert _{L^{2}(B_{r}(y))}\left\Vert D^{2}v\right\Vert
_{L^{2}(B_{r}(y))}+(C_{5}+C_{6}+C_{p}^{\prime}C_{7})\left(  r^{2\alpha}%
r^{n}\right)^{\frac{1}{2}}\left\Vert D^{2}v\right\Vert _{L^{2}(B_{r}(y))}.
\]
That is
\[
\left\Vert D^{2}v\right\Vert _{L^{2}(B_{r}(y))}\leq\frac{1}{\Lambda}\left\{
\Upsilon(r)\left\Vert D^{2}f\right\Vert _{L^{2}(B_{r}(y))}+(C_{5}+C_{6}%
+C_{p}^{\prime}C_{7})\left(  r^{2\alpha}r^{n}\right)  ^{\frac{1}{2}}\right\}
\]
or%
\[
\left\Vert D^{2}v\right\Vert _{L^{2}(B_{r}(y))}^{2}\leq\frac{2}{\Lambda^{2}%
}\left\{  \Upsilon^{2}(r)\left\Vert D^{2}f\right\Vert _{L^{2}(B_{r}(y))}%
^{2}+(C_{5}+C_{6}+C_{p}^{\prime}C_{7})^{2}r^{2\alpha}r^{n}\right\}  .
\]
Using Corollary \ref{Cor2}, for any $0<\rho\leq r$ we get
\begin{align}
\int_{B_{\rho}(y)}  &  \left\vert D^{2}f-(D^{2}f)_{\rho}\right\vert ^{2}%
\leq4C_{2}\left(  \frac{\rho}{r}\right)  ^{n+2}\int_{B_{r}(y)}\left\vert
D^{2}f-(D^{2}f)_{r}\right\vert ^{2}+\left(  8+16C_{2}\right)  \int_{B_{r}%
(y)}\left\vert D^{2}v\right\vert ^{2}\nonumber\\
&  \leq4C_{2}\left(  \frac{\rho}{r}\right)  ^{n+2}\int_{B_{r}(y)}\left\vert
D^{2}f-(D^{2}f)_{r}\right\vert ^{2} +\frac{2}{\Lambda}\left\{  \Upsilon
^{2}(r)\left\Vert D^{2}f\right\Vert _{L^{2}(B_{r}(y))}^{2}+(C_{5}+C_{6}%
+C_{p}^{\prime}C_{7})^{2}r^{2\alpha}r^{n}\right\}  . \label{AA}%
\end{align}
Next, to get decay on the $\left\Vert D^{2}f\right\Vert _{L^{2}}^{2}$
factor, we will find an $r_{0}<1/2$ to be determined, such that for $r<r_{0}$
we have
\[
\int_{B_{\rho}(y)}\left\vert D^{2}f\right\vert ^{2}\leq C_{9}\rho^{n-\delta}
\label{nextgoal}%
\]
where $\delta=1-\tilde{\alpha}<\alpha$. \ In order to do this, first observe
(\ref{2hl}). We may replace $\varepsilon_{0}$ by $\Upsilon^{2}(r)$ by virtue
of (\ref{nbb}). We let $\tilde{\alpha}=1-\delta$, which will result in a
different value $\tilde{q}$ in the derivation leading up to (\ref{2hl}). By
repeating the derivation of (\ref{2hl}) replacing only $\varepsilon_{0}$ by
$\Upsilon^{2}(r)$, $\alpha$ by $\tilde{\alpha}=1-\delta,$ and $K$ by a
$\tilde{K}$ determined by the different norms arising from now the exponent
$\tilde{q}$ $=n/(2-2\tilde{\alpha})$, we get
\begin{align*}
\int_{B_{\rho}(y)}\left\vert D^{2}f\right\vert ^{2}  &  \leq4C_{1}\left(
\frac{\rho}{r}\right)  ^{n}\left\Vert D^{2}f\right\Vert _{L^{2}(B_{r}(y))}%
^{2}+\frac{2\left(  2+8C_{1}\right)  }{\Lambda^{2}}\left(  \Upsilon
^{2}(r)\left\Vert D^{2}f\right\Vert _{L^{2}(B_{r}(y))}^{2}+\tilde
{K}r^{n-2+2\tilde{\alpha}}\right) \\
&  =\left(  4C_{1}\left(  \frac{\rho}{r}\right)  ^{n}+\frac{2\left(
2+8C_{1}\right)  }{\Lambda^{2}}\Upsilon^{2}(r)\right)  \left\Vert
D^{2}f\right\Vert _{L^{2}(B_{r}(y))}^{2}+\frac{2\left(  2+8C_{1}\right)
}{\Lambda^{2}}\tilde{K}r^{n-2+2\tilde{\alpha}}.
\end{align*}
As before, denote 
\begin{align*}
\phi(\rho)  &  =\int_{B_{\rho}(y)}\left\vert D^{2}f\right\vert ^{2}\\
A  &  =4C_{1}\\
\varepsilon &  =\frac{2\left(  2+8C_{1}\right)  }{\Lambda^{2}}\Upsilon
^{2}(r_{0})\\
B  &  =\frac{2\left(  2+8C_{1}\right)  }{\Lambda^{2}} \tilde{K}\\
\alpha &  =n\\
\beta &  =n-2\delta\\
\gamma &  =n-\delta.
\end{align*}
Now by (\ref{nbb1}) and Lemma \ref{HanLin}, there exists $r_{0}$ small enough
such that
\[
\frac{2\left(  2+8C_{1}\right)  }{\Lambda^{2}}\Upsilon^{2}(r_{0}%
)\leq\varepsilon^{\ast}\left(  A,\alpha,\beta,\gamma\right)  , \label{rnot}%
\]
for the $\varepsilon^{\ast}$ provided by Lemma \ref{HanLin}, and we have for
$\rho<r_{0}$%
\begin{align*}
\phi(\rho)  &  \leq C_{8}\left\{  \left(  \frac{\rho}{r}\right)  ^{n-\delta
}\phi(r)+Br^{n-2\delta}\right\} \\
&  \leq C_{8}\frac{1}{r_{0}^{n-\delta}}\rho^{n-\delta}\left\Vert
D^{2}f\right\Vert _{L^{2}(B_{r_{0}})}+B\rho^{n-2\delta}\\
&  \leq C_{9}\rho^{n-\delta}.
\end{align*}

Turning back to (\ref{AA}), we now have, for $r<r_{0}$%
\begin{align*}
\int_{B_{\rho}(y)}\left\vert D^{2}f-(D^{2}f)_{\rho}\right\vert ^{2}  &
\leq4C_{2}\left(  \frac{\rho}{r}\right)  ^{n+2}\int_{B_{r}(y)}\left\vert
D^{2}f-(D^{2}f)_{r}\right\vert ^{2}\\
&  +\frac{2}{\Lambda}\left\{  \Upsilon^{2}(r)C_{9}r^{n-\delta}+(C_{5}%
+C_{6}+C_{p}^{\prime}C_{7})^{2}r^{2\alpha}r^{n}\right\} \\
&  \leq4C_{2}\left(  \frac{\rho}{r}\right)  ^{n+2}\int_{B_{r}(y)}\left\vert
D^{2}f-(D^{2}f)_{r}\right\vert ^{2}+\frac{2}{\Lambda}C_{4}C_{9}r^{2\alpha
+n-\delta}\\
&  +\frac{2}{\Lambda}(C_{5}+C_{6}+C_{p}^{\prime}C_{7})^{2}r^{2\alpha}r^{n}.
\end{align*}
Now we can apply Lemma \ref{HanLin} yet again, this time with
\begin{align*}
\phi(\rho)  &  =\int_{B_{\rho}(y)}\left\vert D^{2}f-(D^{2}f)_{\rho}\right\vert
^{2}\\
A  &  =4C_{2}\\
\alpha &  =n+2\\
B  &  =\frac{2}{\Lambda}\left[  C_{4}C_{9}+(C_{5}+C_{6}+C_{p}^{\prime}C_{7})^{2}\right] \\
\beta &  =n+2\alpha-\delta\\
\gamma &  =n+2\alpha.
\end{align*}
We then conclude that
\[
\int_{B_{r}(y)}\left\vert D^{2}f-(D^{2}f)_{\rho}\right\vert ^{2}\leq C_{10}r^{n+2\alpha-\delta} \label{fg}
\]
for $r<r_{0}$ (and will be necessarily true for $r\in\lbrack r_{0},\frac{1}
{2}]$ as well, perhaps modifying $C_{10}$). \ Noting that this applies for any
$y\in B_{1/4}$ we apply \cite[Theorem 3.1]{HanLin} \ to $D^{2}f$ to conclude
that $D^{2}f\in C^{\left(  2\alpha-\delta\right)  /2}(B_{1/5}).$ Noting
$f=u^{h_{p}}$ we may take a limit and conclude that $u$ must enjoy uniform
$C^{3,\alpha}$ estimates on $B_{1/5}$.
\end{proof}

We now apply the regularity bootstrapping procedure as in \cite{BW1} to obtain smoothness.

\medskip

\noindent\textit{Proof of Theorem \ref{main_2}.} We may scale the estimate
provided in Proposition \ref{reg1} to get $u\in C^{3,\alpha}(B_{r})$ for any
$r<1$. Letting $f=u^{h_{p_{1}}}$ we may apply the dominated convergence theorem while
passing the limit as $h\rightarrow0$ to the equation (\ref{main3}) and
conclude that, for $v=u_{p_{1}}$
\[
\int_{B_{1}}\left(  \beta^{ij,kl}v_{ik}\eta_{jl}+\gamma^{jl}\eta_{jl}-\psi
^{k}\eta_{kp_{1}}-\zeta\eta_{p_{1}}\right)  dx=0
\]
where
\[
\beta^{ij,kl}(x)=\frac{\partial F^{jl}}{\partial u_{ik}}(x,Du,D^{2}u)\in
C^{1,\alpha}(B_{r})
\]
\[
\gamma^{jl}(x)=\frac{\partial F^{jl}}{\partial u_{k}}(x,Du,D^{2}u))f_{k}(x)+\frac{\partial F^{jl}}{\partial x_{p_{1}}}(x,Du,D^{2}u)\in
C^{1,\alpha}(B_{r}).
\]
Noting that the functions $\psi^{k},\zeta$ are $C^{1,\alpha}$ when $u$ in
$C^{3,\alpha},$ we can integrate by parts in the last two terms to get
\[
\int_{B_{1}}\left(  \beta^{ij,kl}v_{ik}\eta_{jl}+\gamma^{jl}\eta_{jl}
+\partial_{x_{p_{1}}}\psi^{k}\eta_{k}+\partial_{x_{p_{1}}}\zeta\eta\right)
dx=0.
\]
Following the difference quotient procedure leading to (\ref{main3}), this
time in the direction $p_{2}$
\[
\int_{B_{1}}\bigg([\beta^{ij,kl}v_{ik}+\gamma^{jl}]^{h_{p}}\eta_{jl}+\partial_{x_{p_{1}}}\psi^{k}\eta_{k}^{-h_{p_{2}}}+\partial_{x_{p_{1}}}\zeta\eta^{-h_{p_{2}}}\bigg)\,dx=0.
\]
Expanding
\[
\int_{B_{1}} \left(  \left(  \beta^{ij,kl}\right)  ^{h_{p_{2}}}v_{ik}+\left(
\gamma^{jl}\right)  ^{h_{p_{2}}}+\left(  \beta^{ij,kl}\right)  v_{ik}^{h_{p_{2}}}\right)  \eta_{jl}+\partial_{x_{p_{1}}}\psi^{k}\eta_{k}^{-h_{p_{2}}}+\partial_{x_{p_{1}}}\zeta\eta^{-h_{p_{2}}}\bigg)\,dx=0.
\]
Observe that each of the terms $\left(  \beta^{ij,kl}\right)^{h_{p_{2}}}v_{ik},\left(  \gamma^{jl}\right)^{h_{p_{2}}},\partial_{x_{p_{1}}}\psi
^{k},\partial_{x_{p_{1}}}\zeta$ are $C^{\alpha}$ with uniform estimates on
$B_{r}.$

So letting
\begin{align*}
\tilde{\gamma}^{jl}  &  =\left(  \beta^{ij,kl}\right)  ^{h_{p_{2}}}v_{ik}+\left(  \gamma^{jl}\right)  ^{h_{p_{2}}}\\
\tilde{\psi}^{k}  &  =\partial_{x_{p_{1}}}\psi^{k}\\
\tilde{\zeta}  &  =\partial_{x_{p_{1}}}\zeta
\end{align*}
we see that $\tilde{v}=v^{h_{p_{2}}}$ satisfies
\begin{equation}
\int_{B_{1}}\left(  \beta^{ij,kl}\tilde{v}_{ik}\eta_{jl}+\tilde{\gamma}^{jl}\eta_{jl}+\tilde{\psi}^{k}\eta_{k}^{-h_{p_{2}}}+\tilde{\zeta}\eta^{-h_{p_{2}}}\right)  dx=0 \label{ummj}
\end{equation}
which is of identical form as equation (\ref{main3}). By our $\Lambda
$-uniform assumption on (\ref{main}), the above equation is uniformly
elliptic, as $\beta$ has not changed. \ Now we apply verbatim the proof of
Proposition \ref{reg1}, noting that all coefficients in sight are H\"{o}lder
continuous, we get $D^{2}\tilde{v}\in C^{\alpha^{\prime}}$. Since $\tilde{v}$
is the difference quotient of a derivative of $u,$ we may take $h\rightarrow0$
and conclude that $u_{p_{1} p_{2}}\in C^{2,\alpha^{\prime}}(B_{r})$ with
estimates for any $\alpha^{\prime}<\alpha,$ for $r<1$, thus $u\in
C^{4,\alpha^{\prime}}(B_{r}).$

Note that when bootstrapping from $C^{m-1,\alpha}$ to $C^{m,\alpha^{\prime}}$
via (\ref{ummj}) for $\tilde{v}= u_{p_{1} p_{2} ... p_{m-3}}^{p_{m-2}} $ we
may take the limit of (\ref{ummj}) to get
\[
\int_{B_{1}}\left(  \beta^{ij,kl}\tilde{v}_{p_{m-2} ik}\eta_{jl}+\tilde
{\gamma}^{jl}\eta_{jl}-\tilde{\psi}^{k}\eta_{kp_{m-2}}-\tilde{\zeta}%
\eta_{p_{m-2}}\right)  dx=0
\]
but now $\tilde{\psi}^{k},\tilde{\zeta}\in C^{1,\alpha}$ so we may integrate
by parts and take another difference quotient in another direction $p_{m-1}$
to obtain another expression very similar to (\ref{ummj}), again with
H\"{o}lder regularity holding for all the coefficients and one higher order of
derivative arising in $\tilde{v}.$ Repeating the proof of Proposition
\ref{reg1}, we conclude $u_{p_{1} p_{2} ... p_{m-1}}\in C^{2,\alpha}(B_{r})$.
\ In this way we can obtain estimates of any order.
\hfill$\square$

\medskip

\noindent\textit{Proof of Theorem \ref{main1:Intro}.} Observing that condition
(\ref{cons1}) is equivalent to condition (\ref{cons}), the result follows
immediately from Theorems \ref{main_1} and \ref{main_2}.
\hfill$\square$

\section{Derivation of the Euler-Lagrange equations on a Riemannian ball}

\label{sec_EL}

We start by deriving the equation for a manifold that is volume stationary
among gradient graphs.

\begin{definition}
Let $\Gamma$ be the set of gradient graphs of functions $u\in C^{1,1}(B_{1})$
with $Du(0)=0$ and $\left\Vert Du \right\Vert _{L^{\infty}} \leq1$, where
$B_{1}\subset\mathbb{R}^{n}$, and
\[
\Gamma(u)=\left\{  \left(  x,Du\left(  x\right)  \right)  :x\in B_{1}\right\}
\subset B_{2}^{2n}.
\]
Let $h$ be a Riemannian metric on the euclidean ball $B_{2}^{2n}$ in
$\mathbb{R}^{2n}$ with $h\left(  0\right)  =\delta_{0}$. We say that
$\Gamma(u)$ is \textit{volume stationary in $(B_{1},h)$ among gradient graphs
in $\Gamma$}, if
\[
\left.  \frac{d}{dt}\mbox{Vol}_{h}(\Gamma(u+t\eta))\right|  _{t=0}=0,
\ \ \ \forall\eta\in C^{\infty}_{c}(B_{1})
\]
where $\mbox{Vol}_{h}$ is volume measured in $h$.
\end{definition}

The volume functional $\mbox{Vol}_{h}$ acting on $\Gamma$ is given by
\[
\mbox{Vol}_{h}(\Gamma(u))=\int_{B_{1}}\sqrt{\det(g_{ij}(x))}\,dx
\]
where, in the standard euclidean basis $\{e_{1},\dots,e_{n},e_{1+n}%
,\dots,e_{2n}\}$ of $\mathbb{R}^{2n}=\mathbb{R}^{n}\times\mathbb{R}^{n}$, the
induced metric $g$ from $h$ on $\Gamma(u)\subset\mathbb{R}^{n}\times
\mathbb{R}^{n}$ is
\begin{align}
g_{ij} &  =h(e_{i}+\sum_{k}u_{ki}e_{k+n},e_{j}+\sum_{l}u_{lj}e_{l+n}%
)\label{induced metric}\\
&  =h_{ij}+\sum_{k}u_{ki}h_{k+n,j}+\sum_{l}u_{lj}h_{l+n,i}+\sum_{k,l}%
u_{ki}u_{lj}h_{l+n,k+n}\nonumber
\end{align}
with $1\leq i,j\leq n$. We may write
\begin{align*}
h_{ij}(x,Du(x)) &  =\delta_{ij}+{\mathcal{A}}_{ij}(x,Du(x))\\
h_{l+n,k+n}(x,Du(x)) &  =\delta_{kl}+{\mathcal{B}}_{kl}(x,Du(x))\\
h_{k+n,j}(x,Du(x)) &  ={\mathcal{C}}_{kj}(x,Du(x)).
\end{align*}
Note that $\mathcal{C}$ need not be symmetric, while $\mathcal{A}$ and
$\mathcal{B}$ are symmetric. In block diagonal form of matrices we have
\begin{equation}
h=\left(
\begin{array}
[c]{cc}%
I & 0\\
0 & I
\end{array}
\right)  +\left(
\begin{array}
[c]{cc}%
\mathcal{A} & \mathcal{C}\\
\mathcal{C}^{T} & \mathcal{B}%
\end{array}
\right)  .\label{defhm}%
\end{equation}
Now we have
\begin{equation}
g_{ij}=\delta_{ij}+u_{ik}\delta^{kl}u_{lj}+{\mathcal{A}}_{ij}+u_{im}%
u_{pj}\delta^{mk}\delta^{pl}{\mathcal{B}}_{kl}+u_{ki}\delta^{kl}{\mathcal{C}%
}_{lj}+u_{kj}\delta^{kl}{\mathcal{C}}_{li}.\label{gdur}%
\end{equation}
Therefore, as a matrix-valued function defined on $(x,Du)$, the induced metric
$g$ is quadratic in $D^{2}u$. In particular,
\begin{equation}
\left\vert \frac{\partial g_{ij}}{\partial u_{kl}}\right\vert \leq C(n,\Vert
h\Vert_{C^{0}})\sup\left\vert D^{2}u\right\vert +C(n)\sup_{i,j}\left\vert
h_{i+n,j}\right\vert ,\label{g derivative}%
\end{equation}
where (and in sequel) we set
\begin{equation}
\Vert h\Vert_{C^{0}}=\sup_{B_{2}}\left\{  \left\vert h_{pq}\right\vert ,1\leq
p,q\leq2n\right\}  .\label{noths}%
\end{equation}
Now, we compute the first variation of $\mbox{Vol}_{h}$. Take a variation
generated by $\eta\in C_{c}^{\infty}(B_{1})$ for the path
\begin{equation}
\gamma\lbrack t](x)=u(x)+t\eta(x),\label{verteta}%
\end{equation}
which varies the manifold $\Gamma(u)$ along the $y$-direction in $B_{2}^{2n}$.
Denote the induced metric from $h$ on $\Gamma(u+t\eta)$ by $g(t)$.
Straightforwardly,
\begin{align*}
g_{ij}(t)=\  &  \delta_{ij}+\left(  u_{ik}+t\eta_{ik}\right)  \delta
^{kl}\left(  u_{lj}+t\eta_{lj}\right)  \\
&  +{\mathcal{A}}_{ij}(x,Du(x)+tD\eta(x))\\
&  +\left(  u_{im}+t\eta_{im}\right)  \left(  u_{pj}+t\eta_{pj}\right)
\delta^{mk}\delta^{pl}{\mathcal{B}}_{kl}(x,Du(x)+tD\eta(x))\\
&  +\left(  u_{ki}+t\eta_{ki}\right)  \delta^{kl}{\mathcal{C}}_{lj}%
(x,Du(x)+tD\eta(x))\\
&  +\left(  u_{kj}+t\eta_{kj}\right)  \delta^{kl}{\mathcal{C}}_{li}%
(x,Du(x)+tD\eta(x)).
\end{align*}

Next, we compute the derivative at $t=0$
\begin{align*}
\left.  \frac{d}{dt}g_{ij}(t) \right|  _{t=0} =  &  \left(  u_{ik}\delta
^{kl}\eta_{lj}+\eta_{ik}\delta^{kl}u_{lj}\right)  +\left(  u_{im}\eta
_{pj}+\eta_{im}u_{pj}\right)  \delta^{mk}\delta^{pl}{\mathcal{B}}%
_{kl}(x,Du(x))\\
&  +\eta_{ki}\delta^{kl}{\mathcal{C}}_{lj}(x,Du(x))+\eta_{kj}\delta
^{kl}{\mathcal{C}}_{li}(x,Du(x))\\
&  +\left\{
\begin{array}
[c]{c}%
D_{y}{\mathcal{A}}_{ij}(x,Du(x))\\
+u_{ki}\delta^{kl}D_{y}{\mathcal{C}}_{lj}(x,Du(x))+u_{kj}\delta^{kl}%
D_{y}{\mathcal{C}}_{li}(x,Du(x))\\
+u_{im}u_{pj}\delta^{mk}\delta^{pl}D_{y}{\mathcal{B}}_{kl}(x,Du(x))
\end{array}
\right\}  \cdot D\eta.
\end{align*}

Then
\begin{align*}
&  \left.  \frac{d}{dt} \mbox{Vol}_{h}(\gamma\lbrack t]) \right|  _{t=0}
=\int_{B_{1}}\frac{1}{2}\sqrt{g[t]}g^{ij}[t]\left.  \frac{d}{dt}g_{ij}[t] dx
\right|  _{t=0}\\
&  =\frac{1}{2}\int_{B_{1}}\sqrt{g}g^{ij}\left(  u_{ik}\delta^{kl}\eta
_{lj}+\eta_{ik}\delta^{kl}u_{lj}+\left(  u_{im}\eta_{pj}+\eta_{im}%
u_{pj}\right)  \delta^{mk}\delta^{pl}{\mathcal{B}}_{kl}(x,Du(x))\right)  dx\\
&  \ \ +\frac{1}{2}\int_{B_{1}}\sqrt{g}g^{ij}\left(  \eta_{ki}\delta
^{kl}{\mathcal{C}}_{lj}(x,Du(x))+\eta_{kj}\delta^{kl}{\mathcal{C}}%
_{li}(x,Du(x))\right)  dx\\
&  \ \ +\frac{1}{2}\int_{B_{1}}\sqrt{g}g^{ij}
\begin{array}
[c]{c}%
\left\{
\begin{array}
[c]{c}%
D_{y}{\mathcal{A}}_{ij}(x,Du(x))\\
+u_{ki}\delta^{kl}D_{y}{\mathcal{C}}_{lj}(x,Du(x))+u_{kj}\delta^{kl}%
D_{y}{\mathcal{C}}_{li}(x,Du(x))\\
+u_{im}u_{pj}\delta^{mk}\delta^{pl}D_{y}{\mathcal{B}}_{kl}(x,Du(x))
\end{array}
\right\}  \cdot D\eta
\end{array}
dx.
\end{align*}
Dropping dependencies for easier presentation, and making use of symmetries%

\begin{align*}
\left.  \frac{d}{dt}\mbox{Vol}_{h}(\gamma[t])\right|  _{t=0}  &  =\int_{B_{1}%
}\sqrt{g}g^{ij}\left(  u_{ik}\delta^{kl}+u_{im}\delta^{mq}\delta
^{lk}{\mathcal{B}}_{qk}\right)  \eta_{lj}dx+\int_{B_{1}}\sqrt{g}g^{ij}%
\eta_{kj}\delta^{kl}{\mathcal{C}}_{li}dx\\
&  +\frac{1}{2}\int_{B_{1}}\sqrt{g}g^{ij}
\begin{array}
[c]{c}%
\left\{  D_{y}{\mathcal{A}}_{ij}+2u_{ik}\delta^{kl}D_{y}{\mathcal{C}}%
_{lj}+u_{im}u_{pj}\delta^{mk}\delta^{pl}D_{y}{\mathcal{B}}_{kl}\right\}  \cdot
D\eta
\end{array}
dx.
\end{align*}

Then we arrive at the Euler-Lagrange equation of $\mbox{Vol}_{h}$ for
variations in $\Gamma$:

\begin{lemma}
For $1\leq i,j,k,l\leq n$, let
\begin{align}
a^{ij,kl}(x,Du,D^{2}u)  &  =\sqrt{g}g^{ij}\delta^{kl}+\sqrt{g}g^{ij}%
{\mathcal{B}}_{lk}\label{l324}\\
b^{jk}(x,Du,D^{2}u)  &  =\sqrt{g}g^{ij}{\mathcal{C}}_{ki}\nonumber\\
c^{k}(x,Du,D^{2}u)  &  =\frac{1}{2}\sqrt{g}g^{ij}\left(  D_{{y^{k}}%
}{\mathcal{A}}_{ij}+2u_{ik}D_{{y^{k}}}{\mathcal{C}}_{kj}+u_{ik}u_{lj}%
D_{{y^{k}}}{\mathcal{B}}_{kl}\right) \nonumber\\
F^{jl}(x,Du,D^{2}u)  &  =a^{ij,kl}u_{ik}+b^{jl} \label{l344}%
\end{align}
Then the Euler-Lagrange equation of $\mbox{Vol}_{h}$ under variations in
$\Gamma$ is
\begin{equation}
\int F^{jl}\eta_{jl}+c^{k}\eta_{k}\,dx=0,\ \ \ \text{for all }\eta\in
C_{c}^{\infty}(B_{1}). \label{EL equation}%
\end{equation}

\end{lemma}

\begin{lemma}
\label{ellipticity} For any $s>0$ there exists $\varepsilon_{1}(s,n)$ $<1$
depending only on $s$ and $n$ such that if
\begin{align*}
h(0) &  =I_{2n}\\
\left\Vert D^{2}u\right\Vert _{L^{\infty}(B_{1})} &  \leq\varepsilon_{1}\\
\left\Vert Dh\right\Vert _{L^{\infty}(B_{1})} &  \leq\varepsilon_{1}%
\end{align*}
all hold we have%
\[
\left\Vert a^{ij,kl}(x,Du(x),D^{2}u(x))-\delta^{ij}\delta^{kl}\right\Vert
_{L^{\infty}(B_{1})}<s.
\]
(Here and below the norms $\left\Vert \cdot\right\Vert$ are defined as in
\eqref{noths}.)\ 
\end{lemma}

\proof From (\ref{l324})
\[
a^{ij,kl}=\delta^{ij}\delta^{kl}+\left(  \sqrt{g}g^{ij}-\delta^{ij}\right)
\delta^{kl}+\sqrt{g}g^{ij}{\mathcal{B}}_{lk}%
\]
It will be convenient to define the following function
\[
\omega(z)=\sup_{M\in S^{n\times n},\left\Vert M\right\Vert \leq z}\left\Vert
\sqrt{\det\left(  I+M\right)  }\left(  1+M\right)  ^{ij}-\delta^{ij}%
\right\Vert
\]
which is clearly continuous for small values of $z$ and vanishes at $z=0.$
\ This allows us to write
\[
\left\Vert a^{ij,kl}-\delta^{ij}\delta^{kl}\right\Vert \leq\omega\left(
\left\Vert g-\delta_{ij}\right\Vert \right)  +(1+\omega\left(  \left\Vert
g-\delta_{ij}\right\Vert \right)  {\mathcal{B}}_{lk}%
\]
Noting from (\ref{defhm})
\begin{align*}
{\mathcal{A}}(0)  &  =0\\
{\mathcal{B}}(0)  &  =0\\
{\mathcal{C}}(0)  &  =0
\end{align*}
and
\[
\sup_{B_{2}^{2n}}\left\{  |\mathcal{A}|,|\mathcal{B}|,|\mathcal{C}|\right\}
\leq 2 \varepsilon_{1}%
\]
we may inspect (\ref{gdur}) and see that
\[
\left\Vert g_{ij}-\delta_{ij}\right\Vert \leq C(n)\left(  \varepsilon
_{1}+3\varepsilon_{1}^{2}+\varepsilon_{1}^{3}\right)
\]
Then%
\[
\left\Vert a^{ij,kl}-\delta^{ij}\delta^{kl}\right\Vert \leq\omega\left(
C(n)\varepsilon_{1}\right)  +\left(  1+\omega\left(  C(n)\varepsilon
_{1}\right)  \right)  \varepsilon_{1}%
\]
Because $\omega$ is continuous near $0$ we choose an $\varepsilon_{1}$ such
that
\[
\left\Vert a^{ij,kl}-\delta^{ij}\delta^{kl}\right\Vert <s.
\]
\endproof

\begin{theorem}
\label{3p1} Suppose that $u(x)$ is a $C^{1,1}$ function on $B_{1}$ such that
$Du=0,D^{2}u(0)=0$ and
\[
\left\Vert D^{2}u\right\Vert _{L^{\infty}(B_{1})}\leq\varepsilon
_{1}(\varepsilon_{0},n)
\]
for $\varepsilon_{1}$ determined by Lemma \ref{ellipticity} and $\varepsilon
_{0}(\frac{1}{2},n)$ determined by (\ref{cons1}). If $\Gamma(u)=\left\{
(x,Du)\right\}  $ is volume stationary among gradient graphs over the
$x$-plane in  for a Riemannian metric $h$ on the euclidean ball $B_{2}^{2n}$
in $\mathbb{R}^{2n}$, then $u$ is smooth in a neighborhood of $0.$
\end{theorem}

\proof We start by perfoming a rescaling. Consider the map%
\[
S:B_{2R}^{2n}\rightarrow B_{2}^{2n}%
\]
given by
\[
S\left(  x,y\right)  =\left(  \frac{x}{R},\frac{y}{R}\right)  .
\]
This gives us a metric $\tilde{h}$ on $B_{2R}^{2n}$ via
\[
\tilde{h}=S^{\ast}h
\]
which satisfies%
\[
\left\Vert D\tilde{h}\right\Vert =\frac{1}{R}\left\Vert Dh\right\Vert .
\]
In particular, by choosing $R$ large, we can scale so that
\[
\left\Vert D\tilde{h}\right\Vert \leq\varepsilon_{1}(\varepsilon_{0}).
\]
Notice that by letting
\[
\tilde{u}=R^{2}u\left(\frac{x}{R}\right)\text{ on }B_{R}%
\]
the gradient graph $\tilde{u}$ is precisely the pullback of the gradient
graph of $u$ via the scaling $S:$%
\begin{align*}
S\left(  x,D\tilde{u}(x)\right) &  =\frac{1}{R}\cdot\left(  x,RDu\left(\frac
{x}{R}\right)\right)  \\
&  =\left(  \frac{x}{R},Du\left(\frac{x}{R}\right)\right)  .
\end{align*}
Note also that
\begin{equation}
D^{2}\tilde{u}(x)=D^{2}u\left(  \frac{x}{R}\right)  \label{rescaleOK}%
\end{equation}
will satisfy the same bounds. \ Now restricting $\tilde{h}$ to $B_{2}^{2n}$
and $\tilde{u}$ to $B_{1}$ we can apply Lemma \ref{ellipticity}, observe
(\ref{l344}) and conclude that the Euler-Lagrange equation (\ref{EL equation})
satisfies the condition in Theorem \ref{main1:Intro}. \ Thus $\tilde{u}$ is
smooth inside $B_{1}.$ Rescaling, we see that $u$ is smooth insde $B_{1/R}.$ \
\hfill $\square$

\section{Hamiltonian stationary Lagrangian submanifolds in a symplectic
manifold}

\label{sec_ main}

Let $(M,\omega)$ be a symplectic $2n$-manifold with a symplectic 2-form
$\omega$, and a Riemannian metric $h$ on $M$ compatible with $\omega$ and an
almost complex structure $J$ on $M$, i.e. $\omega(X,Y)=h(JX,Y)$ for arbitrary
smooth vector fields $X,Y$ on $M$. Suppose that $L$ is a $C^{1}$-regular
submanifold of $M$ which is Lagrangian respect to $\omega$ and Hamiltonian
stationary among all Hamiltonian variations of $L$ fixing the boundary of $L$
if it is non-empty.

In order to arrange for the setting of \ Theorem \ref{3p1},  we adapt a result
from \cite[Prop. 3.2 and Prop. 3.4]{JLS} on existence of Darboux coordinates with
estimates on a symplectic manifold. Let $\pi:\mathcal{U}\rightarrow M$ be the
$U(n)$ frame bundle of $M$. A point in $\mathcal{U}$ is a pair $(p,v)$ with
$\pi(p,v)=p\in M$ and $v:\mathbb{R}^{2n}\rightarrow T_{p}M$ an isomorphism
satisfying $v^{\ast}(\omega_{p})=\omega_{0}$ and $v^{\ast}(h|_{p})=h_{0}$ (the
standard metric on $%
\mathbb{C}
^{n}$). \ The right action of $U(n)$ on  $\mathcal{U}$ is free: $\gamma
(p,v)=(p,v\circ\gamma)$ for any $\gamma\in U(n)$.


\begin{proposition}
[Joyce-Lee-Schoen]\label{Darboux} Let $(M,\omega)$ be a real $2n$-dimensional
symplectic manifold without boundary, and a Riemannian metric $h$ compatible
with $\omega$ and an almost complex structure $J$. Let $\mathcal{U}$ be the
$U(n)$ frame bundle of $M$. Then for small $\varepsilon>0$ we can choose a
family of embeddings $\Upsilon_{p,v}:B^{2n}_{\varepsilon}\rightarrow M$ depending
smoothly on $\left(p,v\right)  \in U$, where $B^{2n}_{\varepsilon}$ is the ball
of radius $\varepsilon$ about $0$ in $\mathbb{C}^{n},$ such that for all
$\left(  p,v\right)  \in U$ we have

\begin{enumerate}
\item $\Upsilon_{p,v}(0)=p$ and $d\Upsilon_{p,v}|_{0}=v:\mathbb{C}%
^{n}\rightarrow T_{p}M;$

\item $\Upsilon_{p,v\circ\gamma}(0)\equiv\Upsilon_{p,v}\circ\gamma$ for all
$\gamma\in U(n);$

\item $\Upsilon_{p,v}^{\ast}(\omega)\equiv\omega_{0}=\frac{\sqrt{-1}}{2}
\sum_{j=1}^{n}dz_{j}\wedge d\bar{z}_{j};$ and

\item $\Upsilon^{*}_{p,v}(h) = h_{0} +O(|z|)$.

\end{enumerate}

Moreover, for a dilation map $\mathbf{t}:B^{2n}_{R}\to B^{2n}_{\varepsilon}$ given by
$\mathbf{t}(z)=tz$ where $t\leq\varepsilon/R$, set $h^{t}_{p,v} =
t^{-2}(\Upsilon_{p,v}\circ\mathbf{t})^{*}h$. Then it holds

\begin{enumerate}
\item[(5)] $\|h^{t}_{p,v}-h_{0}\|_{C^{0}} \leq C_{0} t \ \ \ \ \mbox{and}
\ \ \ \|\partial h^{t}_{p,v}\|_{C^{0}}\leq C_{1} t$,
\end{enumerate}

where norms are taken w.r.t. $h_{0}$ and $\partial$ is the Levi-Civita
connection of $h_{0}$.
\end{proposition}

The following is our main regularity result:

\begin{theorem}
\label{main_geometric} Let $(M,\omega,h)$ be a symplectic manifold. Suppose
that $L$ is a $C^{1}$ Hamiltonian stationary Lagrangian submanifold (possibly
open but without boundary) embedded in $M$. Then $L$ is smooth.
\end{theorem}

\proof Fix an arbitrary point $p$ in $L\subset M$. By Proposition
\ref{Darboux}, we can choose Darboux coordinates around $\Upsilon_{p,v}$ at
$p$, choosing $v$ so that $d\Upsilon_{p,v}|_{0}\left(
\mathbb{R}
^{n}\right)  =T_{p}L.$ Now the submanifold
\[
L_{0}=\Upsilon_{p,v}^{-1}(L\cap\Upsilon_{p,v}\left(  B^{2n}_{\varepsilon}\right)
)\subset B^{2n}_{\varepsilon}\subset\mathbb{C}^{n}%
\]
is Lagrangian and Hamiltonian stationary in $\left(  B^{2n}_{\varepsilon}%
,\Upsilon_{p,v}^{\ast}h,\omega_{0}\right)  $. \ \ As a Lagrangian submanifold
tangent to $%
\mathbb{R}
^{n}$ at $0,$ $L_{0}$ must be represented in a neighborhood of $0$ as the
gradient graph of function $u$ satisfying $Du(0)=0$ and $D^{2}u(0)=0$.
\ \ \ Because $L_{0}$ is $C^{1},$ the Hessian $D^{2}u$ is continuous: We can
choose $0<\varepsilon_{2}<\varepsilon$ if necessary such that
\[
\left\Vert D^{2}u\right\Vert _{C^{0}(B_{\varepsilon_{2}/2}) } <\varepsilon_{1}%
\]
and so that the projection of $L_{0}\cap B^{2n}_{\varepsilon}$ to $%
\mathbb{R}
^{n}$ contains $B_{\varepsilon_{2}/2},$ for the $\varepsilon_{1}$ provided
by Theorem \ref{3p1}. Next we make use of the dilation map in Proposition
\ref{Darboux} (5), choosing $t<\frac{1}{2}\varepsilon_{2},$ small enough so
that
\[
\Vert\partial h_{p,v}^{t}\Vert_{C^{0}}\leq C_{1}t<\varepsilon_{1}.
\]
We now have the following: \ A\ rescaled submanifold $\tilde{L}_{0}$, still
Lagrangian, and Hamiltonian stationary with respect to the metric $h_{p,v}%
^{t},$ which satisfies
\[
\left\Vert Dh_{p,v}^{t}\right\Vert <\varepsilon_{1}%
\]
so that the projection of $\ \tilde{L}_{0}\cap B^{2n}_{\frac{2\varepsilon}{t}}$ to
$%
\mathbb{R}
^{n}$ contains $B_{1}.$  \ \ Noting that the scaling does not change the
Hessian $D^{2}u$ (recall (\ref{rescaleOK})), we see that we are in the setting
of Theorem \ref{3p1}. Since $\omega_{0}$ is the standard symplectic form,
the condition of being Lagrangian Hamiltonian stationary is equivalent to
being critical for gradient graphs.  Theorem \ref{3p1}  now gives us that $u$
is smooth in a neighborhood of $0,$ so $L$ is smooth in a neighborhood of $p.$
As $p$ was arbitrary, $L$ is smooth everywhere.\endproof

\bibliographystyle{amsalpha}
\bibliography{fourth}

\end{document}